\documentclass{article}
\usepackage[utf8]{inputenc}
\usepackage{mathtools}
\DeclarePairedDelimiter{\ceil}{\lceil}{\rceil}
\DeclarePairedDelimiter{\floor}{\lfloor}{\rfloor}
\usepackage{amssymb}
\usepackage{amsthm}
\usepackage[left=1in, right=1in, top=1in, bottom=1in]{geometry}
\usepackage{verbatim}
\normalfont\upshape
\theoremstyle{plain}
\usepackage{fancyhdr}
\usepackage{setspace}
\usepackage{hyperref}

\newtheorem*{classlist*}{Classification}
\newtheorem*{theorem*}{Theorem}
\newtheorem{theorem}{Theorem}[section]
\newtheorem*{lemma*}{Lemma}
\newtheorem{lemma}{Lemma}[section]
\newtheorem*{corollary*}{Corollary}
\newtheorem{corollary}{Corollary}[section]
\newtheorem{proposition}{Proposition}[section]
\theoremstyle{remark}
\newtheorem*{remark*}{Remark}

\newtheorem*{fact*}{Fact}
\newtheorem*{proposition*}{Proposition}
\newtheorem*{definition*}{Definition}
\newtheorem*{remark}{Remark}

\title{Signatures of Multiplicity Spaces in Tensor Products of $\mathfrak{sl}_2$ and $U_q(\mathfrak{sl}_2)$ Representations}
\author{\small Shashwat Kishore and Gus Lonergan}
\date{}

\begin{document}

\maketitle

\begin{abstract}
We study multiplicity space signatures in tensor products of representations of $\mathfrak{sl}_2$ and $U_q(\mathfrak{sl}_2)$, and give some applications. We completely classify definite multiplicity spaces for generic tensor products of $\mathfrak{sl}_2$ Verma modules. This provides a classification of a family of unitary representations of a basic quantized quiver variety, one of the first such classifications for any quantized quiver variety. We use multiplicity space signatures to provide the first real critical point lower bound for generic $\mathfrak{sl}_2$ master functions. As a corollary to this bound, we obtain a simple and asymptotically correct approximation for the number of real critical points of a generic $\mathfrak{sl}_2$ master function. As a first step to quantizing this picture, we obtain a formula for multiplicity space signatures in tensor products of finite dimensional simple $U_q(\mathfrak{sl}_2)$ representations.\end{abstract}

\onehalfspacing

\tableofcontents

\section{Introduction and results}
Let $\mathfrak{g}$ be a semisimple Lie algebra over $\mathbb{C}$, with a choice $\mathfrak{b}$ of Borel subalgebra. The Verma module $M_{\lambda}$ for $\mathfrak{g}$ with real highest weight $\lambda$ carries a certain Hermitian form known as the \textit{Shapovalov form}. It is uniquely determined (up to scalar) by a certain contravariance condition, and is nondegenerate for generic\footnote{In this context, \textit{generic} means lying outside the union of countably many hyperplanes; it is generally possible to write formulas for those hyperplanes.} $\lambda$. Yee has used Kazhdan-Lustzig polynomials to explicitly compute the signatures of these forms (see \cite{YV}).

Given real highest weights $\lambda_1, \ldots, \lambda_n$, the tensor product $\bigotimes_{i=1}^n M_{\lambda_i}$ carries a Hermitian form equal to the product of the Shapovalov forms. If $(\lambda_1, ..., \lambda_n)$ is generic then these forms are nondegenerate, and this tensor product splits as a direct sum:
$
\bigotimes_{i=1}^n M_{\lambda_i} = \bigoplus_{\mu\in P^+} M_{(\sum_i \lambda_i)-\mu} \otimes E_{\mu}
$.
Here $P^+$ is the positive part of the root lattice and $E_{\mu}\cong \operatorname{Hom} (M_{(\sum \lambda_i)-\mu}, \bigotimes_{i=1}^n M_{\lambda_i})$ is the \textit{level $\mu$ multiplicity space}. These summands are orthogonal, and the induced Hermitian form on each isotypic piece is nondegenerate and contravariant; it follows by uniqueness of the Shapovalov form that the induced Hermitian form on the isotypic piece $M_{(\sum_i \lambda_i)-\mu} \otimes E_{\mu}$ is the tensor product of the Shapovalov form on $M_{(\sum_i \lambda_i)-\mu}$ with a certain uniquely determined nondegenerate Hermitian form on $E_{\mu}$. The dimension of $E_{\mu}$ is known; the purpose of this paper is to investigate its signature, and some applications.

We restrict attention to $\mathfrak{sl}_2$ (and, later, $U_q(\mathfrak{sl_2})$). We take the standard basis $E,F,H$ and identify weights with their value on $H$. In this case, the meaning of \textit{generic} is made precise as follows: $\lambda\in \mathbb{R}$ is generic if $\lambda\notin \mathbb{Z}_{\geq0}$, and $(\lambda_1, \ldots, \lambda_n)\in\mathbb{R}^n$ is generic if each of $\lambda_1,\ldots,\lambda_n$ and $\sum_i\lambda_i$ is generic. The Verma module $M_{\lambda}$ is generic if $\lambda$ is; a tensor product of generic Verma modules is a module of the form $\bigotimes_{i=1}^n M_{\lambda_i}$ with $(\lambda_1,\ldots,\lambda_n)$ generic. Also $P^+=2\mathbb{Z}_{\geq 0}$; we will write $\mu=2m$ and reindex the multiplicity spaces so that $\bigotimes_{i=1}^n M_{\lambda_i} = \bigoplus_{m\in \mathbb{Z}_{\geq 0}} M_{(\sum_i \lambda_i)-2m} \otimes E_m$. The Shapovalov form is by definition contravariant in the sense that $(E,F)$ and $(H,H)$ are adjoint pairs with respect to it. We denote by $^*$ the antilinear anti-involution of $U(\mathfrak{sl}_2)$ so determined. Our main results are Theorems \ref{thm:class}, \ref{thm:cpb}, and \ref{thm:qcb}. In particular, we find all the definite multiplicity spaces in any tensor product of generic Verma modules (for $\mathfrak{sl}_2$).
\begin{theorem}\label{thm:class} The list in Appendix \ref{app:class} classifies all definite multiplicity spaces in any tensor product of generic Verma modules. \end{theorem} 

\begin{remark} The classification of definite multiplicity spaces given by Theorem \ref{thm:class} provides a classification of a family of unitary representations of a certain quantized quiver variety, one of the first such classifications for any quantized quiver variety. Indeed let $\mathcal{U}=U(\mathfrak{sl}_2)$ and for $\lambda\in\mathbb{C}$ let $\mathcal{U}_{\lambda}$ denote the quotient of $\mathcal{U}$ by the central character associated to $\lambda$. Notice that the quantum Hamiltonian reduction $\mathcal{A}=(\mathcal{U}_{\lambda_0}\otimes\ldots\otimes\mathcal{U}_{\lambda_n}/(\mathcal{U}_{\lambda_0}\otimes\ldots\otimes\mathcal{U}_{\lambda_n})\mathfrak{sl}_2)^{\mathfrak{sl}_2}$ acts naturally on $\operatorname{Hom}(M_{\lambda_0},M_{\lambda_1}\otimes\ldots\otimes M_{\lambda_n})$. Taking $\lambda_1,\ldots,\lambda_n$ generic reals and $\lambda_0=-2m+\lambda_1+\ldots+\lambda_n$, this $\operatorname{Hom}$ space is $E_m$ as above. $\mathcal{A}$ is an example of a quantized quiver variety (see \cite{BL}) for the star-shaped quiver with one central vertex, $n+1$ other vertices each with one edge to the central vertex, and dimension vector $(2,1,\ldots,1)$. Notice that for real $\lambda$ the antilinear anti-involution $^*$ descends to one on $\mathcal{U}_{\lambda}$, and hence on $\mathcal{A}$, also denoted $^*$; moreover, essentially by definition, this latter $^*$ is the adjoint operator map with respect to the induced form on $E_m$.\end{remark}


Next, we use signatures of multiplicity spaces to study a problem in differential topology. There is a family of $\mathfrak{sl}_2$ \textit{master functions} that arises in the study of Knizhnik-Zamolodchikov equations and spin chain Gaudin models (see \cite{SV}). Namely, for any two positive integers $n,m$ and any two sequences of $n$ real numbers $z = (z_1, ..., z_n)$ and $\lambda = (\lambda_1, ..., \lambda_n)$, there is a a hyperplane arrangement $\mathcal{A}=\bigcup_{i=1}^{m}\bigcup_{j=1}^{n}\{t\in \mathbb{C}^m:t_i=z_j\}\cup\bigcup_{1\leq i<j\leq m}\{t\in \mathbb{C}^m:t_i=t_j\}\subset\mathbb{C}^m$ and a master function $F_{z, \lambda,m} : \mathbb{C}^m-\mathcal{A} \to \mathbb{C}$, defined by
$
F_{z, \lambda, m}(t_1, t_2, ..., t_m) = \operatorname{Disc}(Q) \cdot \prod_{i} |Q(z_i)|^{-\lambda_i}
$,
where $\operatorname{Disc}(Q) = \prod_{i < j} (t_i-t_j)^2$ is the discriminant of $Q(x) = \prod_{i=1}^m (x-t_i)$. A critical point $(t_1,\ldots,t_m)$ of $F_{z, \lambda, m}$ is said to be \textit{real} if the corresponding polynomial $Q(x)$ has real coefficients. Computing the number $N_{z, \lambda, m}$ of real critical points of an arbitrary master function $F_{z, \lambda, m}$ is an open and difficult problem. Recently, Mukhin and Tarasov (see \cite{MT}) have given lower bounds for numbers of real solutions in problems appearing in Schubert calculus by computing signatures of Hermitian forms on the Gaudin model. We study real critical points of the master function using an approach based on Mukhin and Tarasov's work, a Bethe ansatz setup due to Etingof, Frenkel, and Kirillov (see \cite{EFK}), and a Bethe vector characterization due to Feigin, Frenkel, and Rybnikov (see \cite{FFR}). Our second main result, Theorem \ref{thm:cpb}, uses multiplicity space signatures to give a lower bound for the number of real critical points of a generic\footnote{We call a master function generic if its real parameters are generic.} $\mathfrak{sl}_2$ master function.
\begin{theorem}\label{thm:cpb} For a generic $\mathfrak{sl}_2$ master function $F_{z, \lambda, m}$, we have
$
|\operatorname{sgn}(E_m)| \leq N_{z, \lambda, m}
$.
\end{theorem}
\noindent Here $\operatorname{sgn}(E_m)$ denotes the signature of the space $E_m$ in the decomposition of $\bigotimes_{i=1}^n M_{\lambda_i}$.\\


We use this lower bound on $N_{z, \lambda, m}$ together with an upper bound on $N_{z, \lambda, m}$ given by Mukhin and Varchenko (see \cite{MV}) to show that $\operatorname{dim}(E_m) = \binom{m+n-2}{n-2}$ is a good approximation for $N_{z, \lambda, m}$ for $m$ large.
\begin{corollary}\label{cor:cpbsharp} For fixed generic $\lambda$ and $z$ sequences, we have
$
\lim_{m \to \infty} \frac{N_{z,\lambda,m}}{\binom{m+n-2}{n-2}} = 1.
$
\end{corollary}
Finally, we extend our work on signatures to quantum groups. For generic\footnote{We call a complex number on the unit circle generic if it is not a root of unity.} $q\in\mathbb{C}$ on the unit circle, the quantum enveloping algebra $U_q(\mathfrak{sl}_2)$ is the standard $q$-deformation of the enveloping algebra of $\mathfrak{sl}_2$ (see \cite{J}). For each nonnegative integer $a$, there is an $(a+1)$-dimensional simple representation of $U_q(\mathfrak{sl}_2)$, denoted $\widetilde{V}_a$, and this representation carries a Shapovalov form. Moreover, the tensor product $\widetilde{V}_a \otimes \widetilde{V}_b$ of two such representations (which is again a representation) carries an induced contravariant nondegenerate Hermitian form, defined using the Drinfeld coboundary structure (see \cite{D}). We recall the definition. There is a \textit{standard universal $R$-matrix}, defined for our choice of coproduct as
$
R = q^{\frac{H \otimes H}{2}} \sum_{i \geq 0} q^{\binom{i}{2}} \cdot  \frac{(q-q^{-1})^{i}}{[i]!} \cdot F^i \otimes E^i
$,
where $[i] = \frac{q^i - q^{-i}}{q-q^{-1}}$ and $[i]! = [1][2]...[i]$. The matrix $\overline{R}$ of the Drinfeld coboundary structure is defined in terms of the $R$-matrix by $\overline{R} = R(R^{21}R)^{-1/2}$. The form on the tensor product $\widetilde{V}_a \otimes \widetilde{V}_b$ is given by $(v_1 \otimes w_1, v_2 \otimes w_2) = \sum_i (a_iv_1, v_2) \cdot (b_iw_1, w_2)$, where $\overline{R} = \sum_i a_i \otimes b_i$ and the pairings $(v_1, v_2)$ and $(w_1, w_2)$ are computed using the forms on $\widetilde{V}_a$ and $\widetilde{V}_b$. By the cactus axiom (see \cite{KT} for the definition), this construction can be extended to a construction of a unique Hermitian form on the tensor product of any number of these finite dimensional simple $U_q(\mathfrak{sl}_2)$ representations.
Because the operator $\operatorname{twist}\circ\overline{R}$ is an isomorphism $\widetilde{V}_a \otimes \widetilde{V}_b\to\widetilde{V}_b \otimes \widetilde{V}_a$, the form on the tensor product is contravariant (in a sense to be explained later). As in the $\mathfrak{sl}_2$ case, the tensor product decomposes as the direct sum of other such finite-dimensional representations with some multiplicity spaces, each carrying its own induced form. Our third main result, Theorem \ref{thm:qcb}, is a combinatorial formula for the multiplicity space signatures in an arbitrary tensor product of finite dimensional simple $U_q(\mathfrak{sl}_2)$ representations.

\begin{theorem}\label{thm:qcb} We have the decomposition
$
\bigotimes_{i=1}^n \widetilde{V}_{a_i} \cong 
\bigoplus_{m \geq 0} \widetilde{V}_{(\sum_i a_i) - 2m} \otimes \widetilde{E}_m
$,
where
$$
\operatorname{sgn}(\widetilde{E}_m) = \sum_{m_1+m_2+\cdots +m_{n-1} = m} \prod_{j=1}^{n-1} \operatorname{sign} \Big[ \binom{1+\sum_{k=1}^{j+1} a_k - \sum_{k=1}^j m_k}{m_j}_q \binom{\sum_{k=1}^j a_j - \sum_{k=1}^{j-1} m_k}{m_j}_q \binom{a_{j+1}}{m_j}_q \Big].
$$
\end{theorem}

\begin{remark*} \begin{enumerate}\item When $q = 1$, the formula is still valid (even though $1$ is not generic) and gives multiplicity space signatures for a tensor product of generic $\mathfrak{sl}_2$ Verma modules.\ \item The content of this theorem is the case $n=2$; the case of general $n$ follows (e.g. by induction). In particular we make no claim as to the definiteness of these spaces in the quantum case, and suppose the form given above may not be particularly well adapted to that question.\end{enumerate} \end{remark*}\

The paper is organized as follows. In Section \ref{sec:prelims}, we cover definitions and background information. In Section \ref{sec:cpb}, we prove Theorem \ref{thm:cpb} and the approximation given by Corollary \ref{cor:cpbsharp}. In Section \ref{sec:qcb}, we prove Theorem \ref{thm:qcb}. We leave the proof (and indeed the statement) of the classification claimed by Theorem \ref{thm:class} to the appendix.

\section{Preliminaries}\label{sec:prelims}
\subsection{Signature characters}
Let $V$ be a representation of $\mathfrak{sl}_2$ on which $H$ acts diagonally, with finite dimensional weight spaces, and whose weights live in the union of finitely many strings of the form $\lambda-\mathbb{Z}^+$, $\lambda\in\mathbb{R}$. The tensor product of any two such has the same property. Assume furthermore $V$ is endowed with a nondegenerate contravariant Hermitian form - that is to say, one for which the adjointness conditions $H^* = H$, $E^* = F$, and $F^* = E$ hold. Examples include the Verma modules $M_\lambda$ with $\lambda\in\mathbb{R}\backslash\mathbb{Z}^+$ together with their Shapovalov forms. Notice that the weight spaces in any such representation are orthogonal, and that direct sums and tensor products of these representations naturally have the same properties.

To any such representation $V$ of $\mathfrak{sl}_2$, we can associate a \textit{signature character} $\operatorname{ch}_s(V)$, which is a (possibly infinite) sum of powers of the formal symbol $e$ with coefficients in the ring $\mathbb{Z}[s]/(s^2-1)$. Namely, we put $\operatorname{ch}_s(V)=\sum_{\alpha} (a_\alpha+b_\alpha s) e^{\alpha}
$, 
where the sum is taken over all weights $\alpha$. Here $a_\alpha$ is the maximal dimension of a positive definite subspace of the weight space $V_\alpha$, and $b_\alpha$ is the maximal dimension of a negative definite subspace of $V_\alpha$. Thus $a_\alpha+b_\alpha$ is the dimension of $V_\alpha$ and $a_\alpha-b_\alpha$ is the signature of $V_\alpha$.  The signature character respects the direct sum and tensor product: $\operatorname{ch}_s( \bigotimes_i V_i) = \prod_i \operatorname{ch}_s(V_i)$ and $\operatorname{ch}_s( \bigoplus_i V_i) = \sum_i \operatorname{ch}_s(V_i).$

We study the tensor products of generic Verma modules using signature characters. Let $\lambda_1, \lambda_2, ..., \lambda_n$ be generic reals. For the rest of this paper, we will denote $\beta_{\lambda_i} = \operatorname{ch}_s(M_{\lambda_i})$.
\begin{proposition}\label{founder} We have
$
\operatorname{ch}_s(M_{\lambda_i}) = \beta_{\lambda_i} = \begin{cases}
\sum_{j=0}^{\infty} s^je^{\lambda_i-2j} &\mbox{if } \lambda_i < 0, \\
\sum_{j=0}^{\floor{\lambda_i}} e^{\lambda_i-2j} + \sum_{j=\ceil{\lambda_i}}^{\infty} s^{j+\ceil{\lambda_i}}e^{\lambda-2j} &\mbox{if } \lambda_i > 0.
\end{cases}
$
\end{proposition}

Here the notations $\lfloor-\rfloor$ and $\lceil-\rceil$ denote respectively the floor and ceiling functions $\mathbb{R}\to\mathbb{Z}$. They will be used throughout the paper.\

There is a unique decomposition of signature characters
$\prod_{i=1}^n \beta_{\lambda_i} = \sum_{m=0}^{\infty} (a_m+sb_m)\cdot \beta_{\lambda-2m},$
where $\lambda=\sum_{i=1}^{n}\lambda_i$.
This decomposition exactly encodes the tensor product decomposition
$\bigotimes_{i=1}^n M_{\lambda_i} \cong \bigoplus_{m=0}^{\infty} M_{\lambda-2m} \otimes E_m$, in the sense that each multiplicity space $E_m$ has dimension equal to $\binom{m+n-2}{n-2} = a_m+b_m$ and signature equal to $a_m-b_m$. Hence, determining when $E_m$ is definite amounts to determining when $a_m = 0$ or $b_m = 0$.
\subsection{Quantum groups case}
For indeterminate $q$, the algebra $U_q(\mathfrak{sl}_2)$ is generated over $\mathbb{C}[q, q^{-1}, \frac{1}{q-q^{-1}}]$ by $E, F, K, K^{-1}$ with defining relations
$KEK^{-1} = q^2E$, 
$KFK^{-1} = q^{-2}F$, and 
$EF-FE = \frac{K-K^{-1}}{q-q^{-1}}.$
For each $a\in \mathbb{Z}^+$, there is a certain simple $U_q(\mathfrak{sl}_2)$-module $\widetilde{V}_a$. It is free over $\mathbb{C}[q, q^{-1}, \frac{1}{q-q^{-1}}]$ of rank $a+1$, with a basis $\{ v_i \}_{i=0}^{a}$ on which the operators $E$, $F$, and $K$ act by
$Ev_i = [a-i+1]v_{i-1}$,
$Fv_i = [i+1]v_{i+1}$, and
$Kv_i = q^{a-2i}v_i$,
where $[k] = \frac{q^k - q^{-k}}{q-q^{-1}}$. As in the classical case, the representation $\widetilde{V}_a$ carries a unique contravariant nondegenerate Hermitian form $(,)$, as a free $\mathbb{C}[q, q^{-1}, \frac{1}{q-q^{-1}}]$-module, satisfying $(v_0,v_0)=1$. Here by contravariant we mean that the adjointness conditions $E^*=F, F^*=E$ and $K^*=K^{-1}$ hold with respect to the form, and by Hermitian we mean sesquilinear with respect to the involution of the ground ring determined by $z\to\overline{z}$ for $z\in\mathbb{C}$ and $q\to q^{-1}$. Thus it will be usually possible to specialize this picture to $q$ lying on the unit circle. This form is known also as the Shapovalov form.
\begin{proposition}\label{prop:qshap} Under the Shapovalov form $(,)$ and the normalization $(v_0, v_0) = 1$, we have
$
(v_i, v_i) = \binom{a}{i}_q$.\\
Here, $\binom{a}{i}_q = \frac{[a]!}{[i]![a-i]!}$ denotes the quantum binomial coefficient.
\end{proposition}
The tensor product $\bigotimes_{i=1}^n \widetilde{V}_{a_i}$ is a representation of $U_q(\mathfrak{sl}_2)$ under the iteration of the coproduct map defined by $\Delta(E) = E \otimes 1 + K \otimes E$, 
$\Delta(F) = 1 \otimes F + F \otimes K^{-1}$, and
$\Delta(K^{\pm 1}) = K^{\pm 1} \otimes K^{\pm 1}$. Furthermore, the Shapovalov forms and the Drinfeld coboundary induce a nondegenerate contravariant Hermitan form on the representation $\bigotimes_{i=1}^n \widetilde{V}_{a_i}$. This tensor product decomposes into a direct sum of finite dimensional simple $U_q(\mathfrak{sl}_2)$ representations $\widetilde{V}_{a}$ with multiplicities; as before, isotypic pieces are orthogonal, and the uniqueness of the Shapovalov form gives each multiplicity space an induced Hermitian form. In Section \ref{sec:qcb} we outline the proof of Theorem \ref{thm:qcb} which gives a signature formula for these multiplicity spaces.

\section{Critical point bound}\label{sec:cpb}
In this section we use the Bethe ansatz method to derive the lower bound on the number of real critical points of the master function given by Theorem \ref{thm:cpb}. Then we use the bound to derive the asymptotic approximation for the number of real critical points of the master function given by Corollary \ref{cor:cpbsharp}. Throughout this section, fix two generic real sequences $\lambda_1,\ldots,\lambda_n$ and $z_1,\ldots,z_n$. Fix a positive integer $m$ and let $U\subset\mathbb{C}^m$ be given by $U=\{(t_1,\ldots,t_m)|t_i\neq z_k \operatorname{ ~for~ all~ } i,k\}$. We consider two functions on $U$. The first is the $\mathbb{C}[x]$-valued function $Q(x)(t_1,\ldots,t_m)=\prod_{i=1}^m (x-t_i)$ and the second is the $\mathbb{C}$-valued master function:
$$
F_{z, \lambda,m}(t_1, t_2, ..., t_m) = \operatorname{Disc}(Q) \cdot \prod_{k} |Q(z_k)|^{-\lambda_k} = \prod_{i < j} (t_i-t_j)^2 \cdot \prod_{i, k} |t_i-z_k|^{-\lambda_k}.
$$
We will assume that $z_k\neq z_l$ for $k\neq l$.

\subsection{Preliminaries for the proof of Theorem \ref{thm:cpb}}\label{sec:cpbprelims}
We begin with some definitions and preliminary results. 
\begin{definition*}\begin{itemize}\item For any $X \in U(\mathfrak{sl}_2)$ and any $i$ with $1 \leq i \leq n$, define an operator $X_i$ on $\bigotimes_{j=1}^n M_{\lambda_j}$ by $X_i = \underbrace{1 \otimes \cdots \otimes 1}_{i-1 \text{ factors}} \otimes X \otimes \underbrace{1 \otimes \cdots \otimes 1}_{n-i \text{ factors}}$.\
\item For each $i$ and $j$ with $1 \leq i \ne j \leq n$, define the \textit{Casimir tensor} $\Omega_{ij}$ by
$
\Omega_{ij} = E_i  F_j + F_i  E_j + \frac{H_i  H_j}{2}.
$\
\item For each $i$ with $1 \leq i \leq n$, define the \textit{Gaudin model Hamiltonian} ${\mathcal{H}}_i$ by
$
\mathcal{H}_i = \sum_{j \ne i} \frac{\Omega_{ij}}{z_i-z_j}.
$\

\item For any complex number $t$ distinct from all $z_i$, define an operator $Z(t) = \sum_{i=1}^n \frac{H_i}{t-z_i}$.\
\item For any $t$ as above, define an operator $Y(t)$ by $Y(t)= \sum_{i=1}^n \frac{F_i}{t-z_i}$.\ 

\item For $(t_1, t_2, ..., t_m)\in U$ define $b_Q = Y(t_1)Y(t_2) \cdots Y(t_m)v$, where $v$ is the tensor product of the canonical generators of each $M_{\lambda_i}$. \end{itemize}\end{definition*}

\begin{remark*} There is a clash of notation: we have the multiplicity spaces $E_m$, and the operators defined above $E_j$. This should not be cause for confusion.\end{remark*}

We now assume that $(t_1,\ldots,t_m)$ is a (complex) critical point of $F_{z, \lambda, m}$. We state some preliminary results, whose proofs can largely be found in \cite{EFK} and \cite{FFR}.
\begin{lemma}\label{lem:hcom} The Gaudin model Hamiltonians $\mathcal{H}_i$ commute with each other, they act on the multiplicity space $E_m=\operatorname{Hom}(M_{(\sum_i \lambda_i)-2m}, \bigotimes_{i} M_{\lambda_i})$ and they are self-adjoint under the induced Hermitian form on this space.
\end{lemma}
We should clarify what it means for $\mathcal{H}_i$ to act on $E_m=\operatorname{Hom}(M_{(\sum_i \lambda_i)-2m}, \bigotimes_{i} M_{\lambda_i})$. $\mathcal{H}_i$ is defined as an operator on $\bigotimes_{i} M_{\lambda_i}$. It turns out that these operators intertwine the action of $\mathfrak{sl}_2$ and so act on $E_m=\operatorname{Hom}(M_{(\sum_i \lambda_i)-2m}, \bigotimes_{i} M_{\lambda_i})$ by composition.
\begin{proof} See \cite{EFK}. \end{proof}

Notice that $E_m$ is canonically identified with the subspace of $\bigotimes_{i} M_{\lambda_i}$ consisting of all vectors of weight $\sum_i\lambda_i-2m$ and annihilated by $E$. Furthermore that action of $\mathcal{H}_i$ on $E_m$ coincides under this identification with the restriction of the action of $\mathcal{H}_i$ on $\bigotimes_{i} M_{\lambda_i}$. 

\begin{lemma}\label{lem:eb_Q} We have $Eb_Q = 0$.
\end{lemma}

\begin{proof} See \cite{EFK}. \end{proof}

\begin{lemma}\label{lem:ZYcom} We have
$$
[Z(t_a), Y(t_b)] = \frac{2}{t_a-t_b}(Y(t_a)-Y(t_b)).
$$
\end{lemma}

\begin{proof} See \cite{EFK}. \end{proof}

\begin{lemma}\label{lem:hb} We have
$
[\mathcal{H}_i, Y(t_1)Y(t_2) \cdots Y(t_m)]v = \frac{-\lambda_iQ'(z_i)}{Q(z_i)}b_Q.
$
\end{lemma}

\begin{proof} See \cite{EFK}. \end{proof}

As a corollary to Lemma \ref{lem:hb}, we have that $b_Q$ is an eigenvector of each $\mathcal{H}_i$ with eigenvalue
$
\frac{-\lambda_iQ'(z_i)}{Q(z_i)} + \left( \frac{\lambda_i}{2} \sum_{j \ne i}^n \frac{\lambda_j}{z_i-z_j} \right).
$
(This is just the eigenvalue computed in Lemma \ref{lem:hb} plus the $\mathcal{H}_i$-eigenvalue of $v$.)

\begin{lemma}\label{lemyo} The joint eigenvalues of the operators $\mathcal{H}_i$ each have multiplicity $1$, and the eigenvectors of the form $Y(t_1)Y(t_2) \cdots Y(t_m)v$ are all the joint eigenvectors up to scalars.
\end{lemma}

\begin{proof} See \cite{FFR}. \end{proof}

\begin{lemma}\label{bQreal} If the joint eigenvector $b_Q$ has real joint eigenvalue, then the critical point $(t_1, ..., t_m)$ is real. That is, the corresponding polynomial $Q$ has real coefficients.
\end{lemma}

\begin{proof} See Appendix \ref{app:bQreal}. \end{proof}

Notice that the converse is immediate from the remark below Lemma \ref{lem:hb}.

\subsection{Proof of Theorem \ref{thm:cpb}}
We now tie together the results from Section \ref{sec:cpbprelims} to deduce the critical point bound. The multiplicity space $E_m$ is canonically isomorphic to the subspace of all $H$-eigenvectors of $\bigotimes_i M_{\lambda_i}$ of eigenvalue $-2m+\sum_i \lambda_i$ which are annihilated by $E$. Since $Eb_Q = 0$ and $Hb_Q = (-2m+\sum_i \lambda_i)b_Q$, we may view $b_Q$ as an element of $E_m$. Note that the Gaudin Hamiltonians $\mathcal{H}_i$ descend to commuting self-adjoint operators on $E_m$. We know that the joint eigenspaces for the $\mathcal{H}_i$ on $E_m$ are all one-dimensional and spanned by elements $b_Q$ corresponding to critical points of the master function. Moreover those joint eigenspaces with real joint eigenvalue are precisely those corresponding to real critical points. Thus the result follows from the following easy lemma in linear algebra:

\begin{lemma}Let $V$ be a finite-dimensional complex vector space equipped with a non-degenerate Hermitian form, and $\mathcal{H}_i$ be finitely many commuting self-adjoint operators on $V$. Then the signature of $V$ is a lower bound for the maximal number of linearly independent real joint eigenvectors for the operators $\mathcal{H}_i$.\end{lemma}

\subsection{Preliminaries for the proof of Corollary \ref{cor:cpbsharp}}
\begin{definition*} Define $\beta_n$ for integer $n$ by
$$
\beta_{n} = e^{-\epsilon}\beta_{n+\epsilon}
$$
for any $0 < \epsilon < 1$. This is well-defined.
\end{definition*}

\begin{definition*} For any real $\mu$, let $\beta_\mu^-$ be $\beta_\mu$ evaluated at $s=-1$. \end{definition*}

\begin{definition*} For each nonnegative integer $n$, define a polynomial $V_{n}$ by
$$
V_{n}(x) = 1+2x+2x^2+\cdots+2x^{n}.
$$
\end{definition*}

We state the following results without proofs.
\begin{lemma} For an integer $n$, we have
$$
\beta_{n}^- = \begin{cases}
\beta_{-1}^- \cdot e^{n+1} \cdot V_{n+1}(e^{-2}) &\mbox{if } n \geq 0 \\
\beta_{-1}^- \cdot e^{n+1} &\mbox{if } n < 0.
\end{cases}
$$
\end{lemma}

\begin{lemma}\label{lem:finperturbs} For any integer $n_1$ and any negative integer $n_2$, we have
$$
e^{n_1}\beta_{n_2} = \beta_{n_1+n_2}
$$
if $n_1+n_2 < 0$. Also, if $n_1+n_2 \geq 0$, then $e^{n_1}\beta_{n_2}$ can be written as a finite sum of signature characters.
\end{lemma}

\subsection{Proof of Corollary \ref{cor:cpbsharp}}
 Recall that $N_{z, \lambda, m}$ denotes the number of real critical points of $F_{z, \lambda,m}$. From Theorem \ref{thm:cpb}, we know $|\operatorname{sgn}(E_m)| \leq N_m \leq \operatorname{dim}(E_m)= \binom{m+n-2}{n-2}$. We will show, for fixed generic $\lambda_i$'s, that as $m$ approaches infinity, the ratio  $$\frac{|\operatorname{sgn}(E_m)|}{\binom{m+n-2}{n-2}}$$ approaches $1$. This will prove Corollary \ref{cor:cpbsharp}. 
 
The tensor product $M = \bigotimes_{i=1}^n M_{\lambda_i}$ has signature character equal to $\prod_{i=1}^n \beta_{\lambda_i}$. We need to examine the signatures of the multiplicity spaces $E_m$ in the decomposition of $M$, which amounts to examining the coefficients of $\beta_{\lambda-2m}^-$ in the decomposition of $\prod_{i=1}^n \beta_{\lambda_i}^-$. We have
\begin{align*}
\prod_{i=1}^n \beta_{\lambda_i}^- &= e^{\sum_{i=1}^n \{ \lambda_i \}} \cdot \prod_{i=1}^n \beta_{\floor{\lambda_i}}^- \\
&= e^{\sum_{i=1}^n \{ \lambda_i \}} \cdot \left( \prod_{i=1}^p \beta_{-1}^- \cdot e^{\ceil{\lambda_i}} \cdot V_{\ceil{\lambda_i}}(e^{-2}) \right) \cdot \left( \prod_{i=p+1}^n \beta_{-1}^- \cdot e^{\ceil{\lambda_i}} \right) \\
&= e^{\sum_{i=1}^n \ceil{\lambda_i}+\{ \lambda_i \} } \cdot \left( \beta_{-1}^- \right)^n \cdot \left( \prod_{i=1}^p V_{\ceil{\lambda_i}}(e^{-2}) \right) \\
&= e^{\sum_{i=1}^n \ceil{\lambda_i}+\{ \lambda_i \} } \cdot \left( \sum_{m=0}^{\infty} \beta_{-n-2m}^- \cdot (-1)^m \cdot \binom{m+n-2}{n-2} \right)  \cdot \left( \prod_{i=1}^p V_{\ceil{\lambda_i}}(e^{-2}) \right).
\end{align*}
Combining the above computation with Lemma \ref{lem:finperturbs}, we obtain that for all sufficiently large $m$, 
\begin{align}\label{eqn:emlarge}
\operatorname{sgn}(E_m) \cdot (-1)^m = \sum_{i=0}^T \binom{m+n-2-i}{n-2} \cdot (-1)^{i} \cdot c_i,
\end{align}
where $T = \sum_{i=1}^p \ceil{\lambda_i}$ and the $c_i$'s are defined by the polynomial identity
$$
\prod_{i=1}^p V_{\ceil{\lambda_i}}(x) = \sum_{i=0}^{T} c_ix^i.
$$
The RHS in (\ref{eqn:emlarge}) is a polynomial in $m$ of degree $n-2$. Its leading term is
$$
\frac{m^{n-2}}{(n-2)!} \cdot (c_0 - c_1 + \cdots + (-1)^Tc_T).
$$
We will show that this leading term is nonzero and compute it explicitly. We have $\sum_{i=0}^T (-1)^ic_i = \prod_{i=1}^p V_{\ceil{\lambda_i}}(-1)$. But $V_{\ceil{\lambda_i}}(-1)$ is equal to $+1$ or $-1$ for each $i$, so $\sum_{i=0}^T (-1)^ic_i$ also equals $+1$ or $-1$. Therefore, we obtain that $\pm \frac{m^{n-2}}{(n-2)!}$ is the leading term of a polynomial which equals $\operatorname{sgn}(E_m) \cdot (-1)^m$ for all sufficiently large $m$. We also know that $\frac{m^{n-2}}{(n-2)!}$ is the leading term of the polynomial in $m$ defined by the binomial coefficient $\binom{m+n-2}{n-2}$. Hence
$$
\lim_{m \to \infty} \frac{|\operatorname{sgn}(E_m)|}{\binom{m+n-2}{n-2}} = 1
$$
as desired. 

\section{Signatures for the quantum group case}\label{sec:qcb}
In this section we prove Theorem \ref{thm:qcb}. Throughout this section, fix a generic $q$ on the complex unit circle and nonnegative integers $a$ and $b$. We start by proving some preliminary results.

\subsection{Preliminaries for the proof of Theorem \ref{thm:qcb}}
Let $v_0$ and $w_0$ denote the standard highest weight vectors in the simple representations $\widetilde{V}_a$ and $\widetilde{V}_b$ of $U_q(\mathfrak{sl}_2)$, respectively. Recall that
$$
\widetilde{V}_a \otimes \widetilde{V}_b \cong \bigoplus_{m=0}^{\operatorname{min} \{ a,b \} } \widetilde{V}_{a+b-2m}.
$$
For each subrepresentation $\widetilde{V}_{a+b-2m} \subset \widetilde{V}_a \otimes \widetilde{V}_b$, we shall see that there is a unique highest weight vector of the form
$$
v_0 \otimes w_m + \sum_{i=1}^m c_{m,i} \cdot v_i \otimes w_{m-i}, 
$$
where $c_{m,i}$ are scalars. We will call the highest weight vector determined by these scalars the \textit{unit-normalized} highest weight vector. Recall the universal $R$-matrix
$$
R = q^{\frac{H \otimes H}{2}} \sum_{n \geq 0} q^{\binom{n}{2}} \cdot \frac{(q-q^{-1})^n}{[n]!} F^n \otimes E^n.
$$
If $T$ is the twist operator $u\otimes v\mapsto v\otimes u$, then $TR:\widetilde{V}_a\otimes\widetilde{V}_b\to\widetilde{V}_b\otimes\widetilde{V}_a$ is an isomorphism of representations; this gives the braiding on the appropriate tensor category of $U_q(\mathfrak{sl}_2)$-modules. Since the highest weight spaces of $\widetilde{V}_a\otimes\widetilde{V}_b$ are all one-dimensional, we see that the image under $TR$ of the unit-normalized highest weight vector of $\widetilde{V}_{a+b-2m}\subset \widetilde{V}_a\otimes\widetilde{V}_b$ must be a scalar multiple of the unit-normalized highest weight vector of $\widetilde{V}_{a+b-2m}\subset \widetilde{V}_b\otimes\widetilde{V}_a$. In the following two lemmas, we compute this scalar, and the scalars $c_{m,i}$.

\begin{lemma}\label{lem:unorm} In the subrepresentation $\widetilde{V}_{a+b-2m} \subset \widetilde{V}_a \otimes \widetilde{V}_b$, the unit-normalized highest weight vector $u$ is
$$
u=\sum_{i=0}^m c_{m,i} \cdot v_i \otimes w_{m-i}, 
$$
where $$c_{m,i} = (-1)^i \cdot q^{ai-i^2+i} \cdot \frac{\binom{b-m+i}{i}_q}{\binom{a}{i}_q}.$$ 
\end{lemma}

\begin{proof} Certainly any vector $u$ of weight $a+b-2m$ is of the form $\sum_{i=0}^m c_{m,i} \cdot v_i \otimes w_{m-i}$. The condition that $u$ is a highest weight vector is equivalent to:
$$
0 = \Delta(E)u = (E \otimes 1 + K \otimes E)u = \sum_{i=1}^m c_{m,i} \cdot Ev_i \otimes w_{m-i} + c_{m,i-1} \cdot Kv_{i-1} \otimes Ew_{m-i+1}.
$$
Recall that $$Ev_i = [a-i+1]\cdot v_{i-1},$$ $$Kv_{i-1} = q^{a-2i+2}\cdot v_{i-1},$$ $$Ew_{m-i+1} = [b-m+i]\cdot w_{m-i}.$$ Thus the above equation is equivalent to 
$$
[a-i+1]c_{m,i}  + c_{m,i-1} \cdot q^{a-2i+2}[b-m+i]=0  
$$
which is to say
$$
c_{m,i} = (-1) \cdot q^{a-2i+2} \cdot \frac{[b-m+i]}{[a-i+1]} \cdot c_{m,i-1},
$$ for $1\leq i\leq m$.
Equivalently
$$
c_{m,i} = (-1)^i \cdot q^{ai-i^2+i} \cdot \frac{\binom{b-m+i}{i}_q}{\binom{a}{i}_q}c_{m,0}.
$$
We set $c_{m,0}=1$ to get the unit-normalized highest weight vector.
\end{proof}

\begin{lemma}\label{lem:rnorm} With $u$ as above, we have 
$$
TRu=\sum_{i=0}^m c_{m,i}' \cdot w_i \otimes v_{m-i},
$$
where $$c_{m,0}' = (-1)^m \cdot q^{ab/2-am-bm+m^2-m} \cdot \frac{\binom{b}{m}_q}{\binom{a}{m}_q}$$ and
$$
c_{m,i}' = (-1)^i \cdot q^{bi-i^2+i} \cdot \frac{\binom{a-m+i}{i}_q}{\binom{b}{i}_q}\cdot c_{m,0}'.
$$
\end{lemma}

\begin{proof} The second point is immediate from the preceding lemma. In computing $Ru$, the only way to get a multiple of $v_0\otimes w_m$ is by applying the summand $q^{\frac{H\otimes H}{2}}$ of $R$ to the summand $v_0\otimes w_m$ of $u$. Thus $c_{m,m}'=q^{ab/2-am}$. The first claim follows.
\end{proof}

\subsection{Proof of Theorem \ref{thm:qcb}}

Recall that $\widetilde{V}_a$ carries a Shapovalov form, and the induced form on $\widetilde{V}_a\otimes\widetilde{V}_b$ is defined using Drinfeld's unitized $R$-matrix $\overline{R} = R(R^{21}R)^{-1/2}$ rather than $R$. Notice that $R^{12}R=TRTR$ is an automorphism of $\widetilde{V}_a\otimes\widetilde{V}_b$ so is multiplication on each factor $\widetilde{V}_{a+b-2m}$ by some scalar. So for $u$ as above, it will be necessary to compute $T\overline{R}(u)$. By the above calculation,
$$
TRTRu=(-1)^m \cdot q^{ab/2-am-bm+m^2-m} \cdot \frac{\binom{b}{m}_q}{\binom{a}{m}_q}\cdot (-1)^m \cdot q^{ab/2-am-bm+m^2-m} \cdot \frac{\binom{a}{m}_q}{\binom{b}{m}_q}u=q^{ab-2am-2bm+2m^2-2m}u.
$$
So $T\overline{R}u=q^{-ab/2+am+bm-m^2+m}\sum_{i=0}^m c_{m,i}' \cdot w_i \otimes v_{m-i}$.
There is a natural Hermitian contravariant pairing $\langle~,\rangle$ between $\widetilde{V}_b\otimes\widetilde{V}_a$ and $\widetilde{V}_a\otimes\widetilde{V}_b$, and the form $(~,)$ on $\widetilde{V}_a\otimes\widetilde{V}_b$ is defined by $(x,y)=\langle T\overline{R}x,y\rangle$. We wish to compute $(u,u)$. We have:
$$
(u,u)  = \langle q^{-ab/2+am+bm-m^2+m}\sum_{i=0}^m c_{m,i}' \cdot w_i \otimes v_{m-i},\sum_{i=0}^m c_{m,i} \cdot v_i \otimes w_{m-i}\rangle\\ 
$$

\noindent We know that under the Shapovalov form on $\widetilde{V}_a$, $(v_i,v_j)=\delta_{ij}\binom{a}{i}_q$. Therefore:
$$
 (u,u) = q^{-ab/2+am+bm-m^2+m}\sum_{i=0}^mc_{m,m-i}'\overline{c_{m,i}}\binom{a}{i}_q\binom{b}{m-i}_q.
 $$
Plugging in the known quantities $c_{m,i}$, $c_{m,i}'$, this comes to 
$$
\frac{\binom{b}{m}_q}{\binom{a}{m}_q}q^{bm-m^2+m}\sum_{i=0}^m q^{-(a+b+2-2m)i} \cdot \binom{b-m+i}{i}_q\binom{a-i}{m-i}_q
$$
We apply the identity $\binom{l_1}{l_2}_q = (-1)^{l_2}\binom{l_2-l_1-1}{l_2}$ twice to get
$$
(-1)^m\frac{\binom{b}{m}_q}{\binom{a}{m}_q}q^{bm-m^2+m}\sum_{i=0}^m q^{-(a+b+2-2m)i} \cdot \binom{m-b-1}{i}_q\binom{m-a-1}{m-i}_q
$$
The quantum Vandermonde identity (see \cite{BC}) states that 
$$
q^{bm-m^2+m}\sum_{i=0}^m q^{-(a+b+2-2m)i} \cdot \binom{m-b-1}{i}_q\binom{m-a-1}{m-i}_q=\binom{2m-a-b-2}{m}_q
$$
and so 
$$
(u,u)=(-1)^m\frac{\binom{b}{m}_q}{\binom{a}{m}_q}\binom{2m-a-b-2}{m}_q.
$$
Applying the earlier identity a third time, we conclude:
$$
(u,u)=\frac{\binom{b}{m}_q}{\binom{a}{m}_q}\binom{a+b+1-m}{m}_q.
$$
Theorem \ref{thm:qcb} follows. 

\section{Further work}
We can ask the same questions for other semisimple Lie algebras, or more generally for Kac-Moody algebras. In particular, we can ask for a definite multiplicity space classification for any Kac-Moody algebra. This will give classifications of families of unitary representations for more quantized quiver varieties.

In addition, it would be interesting to further explore our bound on the number of real critical points of the master function. Specifically, we are interested in how tight this bound is. Computer testing suggests that the bound is good for many $\lambda$ and $z$ sequences, in the sense that the ratio $\frac{N_m}{|\operatorname{sgn}(E_m)|}$ is usually close to $1$. When $m$ is small, however, the bound is bad at least in some special cases. For example, when $m=2$ we can provide a geometric bound for $N_m$ in the following way:\

Consider the real part of the domain of the master function (for $m=2$). It is equal to the complement in $\mathbb{R}^2$ of the lines $t_1=t_2$ and $t_1=z_i$, $t_2=z_i$ for $1\leq i\leq n$. In this way it is the union of square and triangular regions. We compute (depending on the values of the $\lambda_i$) the limits of the master function at the boundary of each connected region, and observe that if this limit is everywhere zero then there is a maximum in that region; and if it is everywhere infinite then there is a minimum in that region; and if there are precisely two connected components of the boundary on which the limit is infinite then there is a saddle point in that region. In any case we obtain a real critical point. We thus obtain a lower bound on the number of real critical points, given as some explicit function of $\lambda_1,\ldots,\lambda_n$.


Finally, it would be interesting to extend the relationship between signatures of multiplicity spaces in tensor products of Verma modules and numbers of real critical values of the master function to the quantum case; associated to the quantum Knizhnik-Zamolodchikov equations should be a quantum master function, and we expect that, as in the classical case, the number of its real critical points is bounded by the signature of the appropriate multiplicity space in an appropriate tensor product of Verma modules for $U_q(\mathfrak{sl}_2)$.


\section{Acknowledgements}
The authors thank Professor Pavel Etingof of MIT for introducing us to this problem and for his helpful conversations and suggestions, in particular the suggestion of using Mukhin and Tarasov's Bethe ansatz. They also thank the Center for Excellence in Education, Massachusetts Institute of Technology, Research Science Institute and MIT-PRIMES program for supporting us in this research.

\appendix

\section{Classification List}\label{app:class}

In this appendix we state the classification of definite multiplicity spaces in
$$
M_{\lambda_1} \otimes M_{\lambda_2} \otimes \cdots \otimes M_{\lambda_n},
$$
as promised in Theorem \ref{thm:class}. Let $\lambda = \sum_{i=1}^n \lambda_i$, and assume the $\lambda_i$'s are generic reals in decreasing order, with the first $p$ of them positive ($0 \leq p \leq n$) and the rest negative. Let us give the representation $\bigotimes_{i=1}^n M_{\lambda_i}$ a name, once and for all: call it $M$. By Proposition \ref{founder}, we have that the signatures of the multiplicity spaces depend only on the values $\floor{\lambda}, \floor{\lambda_1}, ... , \floor{\lambda_n}$. We define the \textit{explicit type} of $M$ to be $\langle \floor{\lambda}, \floor{\lambda_1}, ..., \floor{\lambda_n} \rangle$ and the \textit{implicit type} to be $\langle \floor{\lambda_1}, ..., \floor{\lambda_n} \rangle$. \\

\begin{classlist*}\
\begin{itemize}
\item[Case 1.] $p = 0$ \\ Every space is definite. The even-level spaces are positive definite and the odd-level spaces are negative definite.
\item[Case 2.] $n = 2$ \\ Every space is definite. The sign of the level $m$ space is given by the function $g(\lambda_1, \lambda_2, m)$, defined in subsection \ref{subsec:g}.
\item[Case 3.] $p = 1, n \geq 3$ \\
The definite spaces are all those with levels less than or equal to $\operatorname{max}\{0, \ceil{\frac{\lambda}{2}}\}$. In this range, the even-level spaces are positive definite and the odd-level spaces are negative definite.
\item[Case 4.] $p = n-1, n \geq 3$
\begin{itemize}
\item[a.] If $\lambda < 0$, then the definite spaces are all those with levels less than or equal to $\ceil{\lambda_p}$, and they are all positive definite.
\item[b.] If $\lambda > 0$ and $\ceil{\lambda+1} \leq \ceil{\lambda_p}$, then the definite spaces are all those with levels either equal to $0$ or between $\ceil{\lambda+1}$ and $\ceil{\lambda_p}$ (inclusive). They are all positive definite.
\item[b.] If $\lambda > 0$ and $\ceil{\lambda+1} > \ceil{\lambda_p}$, then the (positive definite) level $0$ space is the unique definite space, unless $M$ has explicit type $\langle 1,0,0,-1 \rangle$, in which case the (negative definite) level $2$ space is the only additional definite space.
\end{itemize}

\item[Case 5.] $p = n, n \geq 3$ \\
The spaces of levels less than or equal to $\ceil{\lambda_p}$) are positive definite. These are all definite spaces, outside of the following exceptional explicit types (for which we give all additional definite spaces): 
\begin{itemize}
\item For explicit type $\langle 3d,d,d,d \rangle$ where $d \geq 0$, the level $2d+1$ and $2d+2$ spaces are positive definite.
\item For explicit type $\langle 3d+2,d,d,d \rangle$ where $d \geq 0$, the level $2d+2$ and $2d+3$ spaces are negative definite.
\item For explicit type $\langle 3d-1,d,d,d-1 \rangle$ where $d \geq 1$, the level $2d+1$ space is positive definite.
\item For explicit type $\langle 3d+1, d,d,d-1 \rangle$ where $d \geq 1$, the level $2d+2$ space is negative definite.
\item For explicit type $\langle 3d-2,d,d-1,d-1 \rangle$ where $d \geq 1$, the level $2d$ space is positive definite.
\item For explicit type $\langle 3d,d,d-1,d-1 \rangle$ where $d \geq 1$, the level $2d+1$ space is negative definite.
\item For explicit type $\langle 3,0,0,0,0 \rangle$, the level $3$ space is negative definite.
\item For explicit type $\langle 4,1,1,1,1 \rangle$, the level $4$ space is positive definite.
\item For explicit type $\langle 0,0,0,...,0 \rangle$ where $p \geq 4$, the level $2$ space is positive definite.
\item For explicit type $\langle 1,1,0,...,0 \rangle$ where $p \geq 4$, the level $2$ space is positive definite.
\end{itemize}

\item[Case 6.] $2 \leq p \leq n-2, n \geq 4$ \\ There is one definite space. It is the level $0$ space and it is positive definite.

\end{itemize} 
\end{classlist*}

\subsection{Definition of the function $g$}\label{subsec:g}

We define the function $g$ on a triple consisting of two nonintegral reals $\lambda_1 > \lambda_2$ and a nonnegative integer $m$ as follows: \\ \\
If $0 > \lambda_1 > \lambda_2$ then
$$
g(\lambda_1, \lambda_2, m) = (-1)^m.
$$
If $\lambda_1 > 0$, $\lambda_2 < 0$, $\lambda_1+\lambda_2<0$, then
$$
g(\lambda_1, \lambda_2, m) =\left\{\begin{matrix}
1 & \lambda_1>m-1\\ 
(-1)^{\lfloor{\lambda_1-m+1}\rfloor} & \lambda_1<m-1
\end{matrix}\right.
.$$
If $\lambda_1 > 0$, $\lambda_2 < 0$, $\lambda_1 + \lambda_2 > 0$ then
\begin{multline*}
g(\lambda_1, \lambda_2, m) = \left\{
\begin{array}{rr}
    (-1)^m & : 0 \leq m \leq \floor{\frac{\lambda_1+\lambda_2}{2}} \\
    (-1)^{\ceil{\frac{\lambda_1+\lambda_2}{2}}} & : \ceil{\frac{\lambda_1+\lambda_2}{2}} \leq m \leq \floor{\frac{\lambda_1+\lambda_2+1}{2}} \\  
    (-1)^{\ceil{\frac{\lambda_1+\lambda_2}{2}}+\ceil{\frac{\lambda_1+\lambda_2+1}{2}} + m} & : \ceil{\frac{\lambda_1+\lambda_2+1}{2}} \leq m \leq \ceil{\lambda_1+\lambda_2} \\ 
    1 & : \ceil{\lambda_1+\lambda_2}+1 \leq m \leq \ceil{\lambda_1} \\ 
    (-1)^{\ceil{\lambda_1}+m} & : \ceil{\lambda_1}+1 \leq m
\end{array}
\right\}
\end{multline*}
If $\lambda_1 > \lambda_2 > 0$, then
\begin{multline*}
g(\lambda_1, \lambda_2, m) = \left\{
\begin{array}{rr}
    1 & : 0 \leq m \leq \floor{\lambda_2} \\ 
    g(\lambda_1, \lambda_2-2\ceil{\lambda_2}, m - \ceil{\lambda_2}) & : \ceil{\lambda_2} \leq m \leq \text{max}(\ceil{\lambda_2}, \floor{\lambda_1}) \\ 
    g(\lambda_1, \lambda_2-2\ceil{\lambda_2}, m - \ceil{\lambda_2}) \\+ 2\floor{\lambda_1}+2\floor{\lambda_2}-2\floor{\lambda_1+\lambda_2} & : \text{max}(\ceil{\lambda_2}, \floor{\lambda_1})+1 \leq m \leq \ceil{\lambda_1}+\ceil{\lambda_2} \\ 
    (-1)^{m-\ceil{\lambda_1}-\ceil{\lambda_2}} & : \ceil{\lambda_1} + \ceil{\lambda_2}+1 \leq m 
\end{array}
\right\}
\end{multline*}

The right way to think about these formulas is to understand that the signs of the (in this case one-dimensional) multiplicity spaces pass through several phases, during each of which they are either alternating or constant; the formulas above are just describing when the transitions between these phases happen.\\

In the following appendices, we present the proof of this classification. 

\section{Proof of Case 1}\label{app:allnegativecase} 
For $\lambda_1, \lambda_2,..., \lambda_n < 0$, we have by Proposition \ref{founder} that $\operatorname{ch}_s(M_{\lambda_i}) = \sum_{j=0}^{\infty} s^je^{\lambda_i-2j} = \frac{e^{\lambda_j}}{1-se^{-2}}$. Hence the signature character of $\bigotimes_{i=1}^n M_{\lambda_i}$ is $\frac{e^{\sum_{i=1}^n \lambda_i}}{(1-se^{-2})^n}$. Using a common generating function manipulation, we get that 
\begin{align*}
\frac{e^{\sum_{i=1}^n \lambda_i}}{(1-se^{-2})^n} &= e^{\sum_{i=1}^n \lambda_i} \sum_{j=0}^{\infty} \binom{j+n-1}{n-1} s^je^{-2j} \\
&= \sum_{j=0}^{\infty} \binom{j+n-2}{n-2} \cdot s^j\beta_{\sum_{i=1}^n \lambda_i - 2j}
\end{align*}
whence we see that all even-level spaces are positive definite and all odd-level spaces are negative definite.

\section{Proof of Case 2}\label{app:n2class}
Since every multiplicity space has dimension $1$, every multiplicity space is definite. We just need to determine which spaces are positive definite and which are negative definite, which we do using the following result (Lemma \ref{lem:edecomp1}). Recall that $\beta_\mu$ denotes the signature character of $M_\mu$; we will write $\beta_\mu^-$ for $\beta_\mu$ evaluated at $s=-1$.
\begin{lemma}\label{lem:edecomp1} We have
$$
e^{\mu} = \beta^-_{\mu} - \operatorname{sign}(\mu)\beta^-_{\mu-2} + (\operatorname{sign}(1-\mu) - 1)\beta^-_{\mu-2\ceil{\mu}}.
$$
\end{lemma}
\begin{proof} This is a straightforward calculation. One should check in turn in each of the intervals $\mu > 1$, $1 > \mu > 0$, and $0 > \mu$. \end{proof} We proceed by analyzing in turn the cases $p=1$ (i.e. $\lambda_1>0,\lambda_2<0$) and $p=2$ (i.e. $\lambda_1>2, \lambda_2>0$).

\subsection{Proof of Case 2 for $p=1$}\label{app:p1n2class} 
If $\lambda_1+\lambda_2 < 0$, then $e^{\mu}\beta_{\lambda_2} = \beta_{\mu+\lambda_2}$ for any $\mu=\lambda_1,\lambda_1-1,\ldots$ (compare with, e.g. Lemma \ref{lem:finperturbs}). In that case, the result is easy. So we assume that $\lambda=\lambda_1+\lambda_2 > 0$. Then we have
\begin{align*}
\beta_{\lambda_1}^-\beta_{\lambda_2}^- &= (e^{\lambda_1}+e^{\lambda_1-2}+ \cdots +e^{\lambda_1-2\floor{\lambda_1}} + \beta_{\lambda_1-2\ceil{\lambda_1}}^-)\beta_{\lambda_2}^- && \text{by~Proposition~}\ref{founder}\\
&= (e^{\lambda_1}+e^{\lambda_1-2}+ \cdots +e^{\lambda_1-2\floor{\lambda_1}})\beta_{\lambda_2}^- + \sum_{m = \ceil{\lambda_1}}^{\infty} (-1)^{m+\ceil{\lambda_1}}\beta_{\lambda-2m}^-  && \text{by Case }1
\end{align*}
At this point, for brevity, we will treat the case $\lfloor\lambda\rfloor\equiv2\mod4$; the other cases are similar, and left as an exercise. Notice that $(e^{\lambda_1}+e^{\lambda_1-2})\beta_{\lambda_2}^-=e^\lambda$, which in turn (by Lemma \ref{lem:edecomp1}) is equal to $\beta_\lambda^--\beta_{\lambda-2}^--2\beta_{\lambda-2\lceil\lambda\rceil}^-$ (for $\lambda>1$). Similarly, we obtain that 
$$
(e^{\lambda_1}+e^{\lambda_1-2}+\ldots+e^{\lambda_1-\lfloor{\lambda}\rfloor})\beta_{\lambda_2}^-=\sum_{m=0}^{\lfloor\lambda\rfloor/2}(-1)^m\beta_{\lambda-2m}^- - \sum_{j=0}^{(\lfloor\lambda\rfloor-2)/4}2\beta_{\lambda-2\lceil\lambda\rceil+4j}^-
$$
while
$$
(e^{\lambda_1-\lfloor{\lambda}\rfloor-2}+e^{\lambda_1-\lfloor{\lambda}\rfloor-4}+\ldots+e^{\lambda_1-2\lfloor{\lambda_1}\rfloor})\beta_{\lambda_2}^-=\sum_{m=\floor{\lambda}/2+1}^{\lfloor\lambda_1\rfloor}\beta_{\lambda-2m}^- .
$$

We thus obtain
$$
\beta_{\lambda_1}^-\beta_{\lambda_2}^- = \sum_{m=0}^{\floor{\lambda}+1}(-1)^m\beta_{\lambda-2m}^-+\sum_{m=\floor{\lambda}+2}^{\floor{\lambda_1}}\beta_{\lambda-2m}^-+\sum_{m=\ceil{\lambda_1}}^{\infty}(-1)^{m+\ceil{\lambda_1}}\beta_{\lambda-2m}^-.
$$
Compare signs with the function $g$ defined above.

\subsection{Proof of Case 2 for $p=2$}\label{app:p2n2class}
For any positive real $t$, write $L_t = e^t + e^{t-2} + \cdots + e^{t-2\floor{t}}$. We have
\begin{align*}
\beta_{\lambda_1}^-\beta_{\lambda_2}^- &= (L_{\lambda_1}+\beta_{\lambda_1-2\ceil{\lambda_1}}^-)(L_{\lambda_2}+\beta_{\lambda_2-2\ceil{\lambda_2}}^-) \\ &= \beta_{\lambda_1}^-\beta_{\lambda_2-2\ceil{\lambda_2}}^- + \beta_{\lambda_2}^-\beta_{\lambda_1-2\ceil{\lambda_1}}^- - \beta_{\lambda_1-2\ceil{\lambda_1}}^-\beta_{\lambda_2-2\ceil{\lambda_2}}^- + L_{\lambda_1}L_{\lambda_2},
\end{align*}
and
\begin{align*}
L_{\lambda_1}L_{\lambda_2} &= \sum_{m=0}^{\floor{\lambda_2}}(m+1)e^{\lambda-2m} \\
&+ \sum_{m=\ceil{\lambda_2}}^{\floor{\lambda_1}} \ceil{\lambda_2}e^{\lambda-2m} \\
&+ \sum_{m=\ceil{\lambda_1}}^{\ceil{\lambda_1}+\ceil{\lambda_2}} (\floor{\lambda_1}+\floor{\lambda_2}-m+1)e^{\lambda-2m}.
\end{align*}
Rewriting the powers of $e$ in $L_{\lambda_1}L_{\lambda_2}$ as signature characters using Lemma \ref{lem:edecomp1} and applying the result of Subsection \ref{app:p1n2class} to the other terms gives the result.

\section{Proof of Case 3}\label{app:nnpclass}
Our strategy is first to identify a finite range of possible definite spaces. To that end, it is convenient to consider the completion $X$ of the group algebra $\mathbb{Z}[s][\mathbb{Z}]=\bigoplus_{j=-\infty}^{\infty}\mathbb{Z}[s].e^j$ given by\begin{align*}X=\bigoplus_{j=0}^{\infty}\mathbb{Z}[s].e^j\oplus \prod_{j=-\infty}^{-1}\mathbb{Z}[s].e^j\end{align*}with the algebra structure determined by continuity. One may readily check (by an upper-triangularity phenomenon) that \begin{align*}X=\bigoplus_{j=0}^{\infty}\mathbb{Z}[s].\beta_j\oplus \prod_{j=-\infty}^{-1}\mathbb{Z}[s].\beta_j.\end{align*} In fact, the ring $X$ has been implicitly where our calculations have taken place thus far. Now notice that $X$ contains the semiring $X^+=\bigoplus_{j=0}^{\infty}\mathbb{Z}^+[s].\beta_j\oplus \prod_{j=-\infty}^{-1}\mathbb{Z}^+[s].\beta_j$. We may therefore introduce a partial order on $X$ by setting $x\leq y$ iff $y-x\in X^+$. In that case we will say that $y$ \emph{contains} $x$. Multiplication by $X^+$ preserves this partial order.

\begin{lemma}\label{lem:nnpbound} No space with level more than $\ceil{\frac{\lambda}{2}}$ is definite.
\end{lemma} 
\begin{proof} We treat the case, $\lambda>0$, the other case being similar. \\
Fix an integer $d\geq \ceil{\lambda/2}$. From Case $1$, we have $\beta_{\lambda_2}\beta_{\lambda_3}\ldots\beta_{\lambda_n}=\sum_{m=0}^{\infty}s^m\binom{m+n-2}{n-2}\beta_{\lambda-\lambda_1-2m}$, which in particular contains either $\beta_{\lambda-\lambda_1-2d}+s\beta_{\lambda-\lambda_1-2d-2}$ or $s\beta_{\lambda-\lambda_1-2d}+\beta_{\lambda-\lambda_1-2d-2}$. We assume the former (without loss of generality). Multiplying by $\beta_{\lambda_1}$, we have that $\beta_{\lambda_1}\beta_{\lambda_2}\ldots\beta_{\lambda_n}$ contains $\beta_{\lambda_1}(\beta_{\lambda-\lambda_1-2d}+s\beta_{\lambda-\lambda_1-2d-2})$, which in turn (using Case $2$) contains $\beta_{\lambda-2d-2}+s\beta_{\lambda-2d-2}$. (The crucial point here is that $\lambda_1>0$, so that the second term of $\beta_{\lambda_1}\beta_{\lambda-\lambda_1-2d}$ has coefficient $1$, rather than $s$). We conclude that $\beta_{\lambda_1}\beta_{\lambda_2}\ldots\beta_{\lambda_n}$ contains $(1+s)\beta_{\lambda-2d-2}$ for any $d\geq\ceil{\lambda/2}$, which completes the proof.\end{proof}

To complete Case $3$, we must show that the multiplicity spaces of levels $m\leq\ceil{\lambda/2}$ are all definite, and have sign $(-1)^m$. This follows by an analysis similar to the proof of Lemma \ref{lem:nnpbound} (in particular first expanding $\beta_{\lambda_2}\beta_{\lambda_3}\ldots\beta_{\lambda_n}$ and then using Case $2$ to multiply the result by $\beta_{\lambda_1}$). We leave this as an exercise.

\section{Proof of Case 4 for $n = 3$}\label{app:nppclass}
We use essentially the same strategy, only the details are more complicated.
\subsection{Proof of Case 4 for $n = 3$ and $\lambda < 0$}\label{subsec:ln}
\begin{lemma}\label{lem:ln1} No space with level greater than $\ceil{\lambda_2}$ is definite.
\end{lemma}
\begin{proof} Since $\lambda_2+\lambda_3 < 0$, the previously proven Case 2 classification gives that the product of $\beta_{\lambda_2}$ and $\beta_{\lambda_3}$ contains the product of $\beta_{\lambda_2-2\ceil{\lambda_2}}$ and $\beta_{\lambda_3}$. Hence, $\beta_{\lambda_1}\cdot \beta_{\lambda_2} \cdot \beta_{\lambda_3}$ contains $\beta_{\lambda_1}\cdot \beta_{\lambda_2-2\ceil{\lambda_2}} \cdot \beta_{\lambda_3}$. By the $n=3$ subcase of Case 3 (previously proven) we get the claim.
\end{proof}

\begin{lemma}\label{lem:ln2} Every space with level at most $\ceil{\lambda_2}$ is positive definite.
\end{lemma}
\begin{proof} From the previously proven Case 2 classification, we obtain that in the decomposition of $M_{\lambda_1} \otimes M_{\lambda_2}$, the spaces with level at most $\ceil{\lambda_2}$ are all positive definite. Applying the Case 2 classification again, we get that if the following condition is satisfied for $0 \leq k \leq \ceil{\lambda_2-1}$:
$$
\lambda_1+\lambda_2 - 2k - 2\ceil{\lambda_1+\lambda_2 - 2k} \leq \lambda_1 + \lambda_2 - 2\ceil{\lambda_2}
$$
then the claim holds. This is equivalent to:
$$
\ceil{\lambda_1+\lambda_2} \geq 2\ceil{\lambda_2}-1
$$
which is true.
\end{proof}

\begin{proof}[Proof of $\lambda < 0$ subcase of Case 4 for $n = 3$] Lemma \ref{lem:ln1} shows that there are no definite spaces with level more than $\ceil{\lambda_2}$. Lemma \ref{lem:ln2} shows that all spaces with level at most $\ceil{\lambda_2}$ are positive definite. Thus, the $\lambda < 0$ case of Case 4 for $n = 3$ holds.
\end{proof}

\subsection{Proof of Case 4 for $n=3$ and $\lambda > 0$}\label{subsec:lp}
From the previously proven Case 2 classification, the product $\beta_{\lambda_1}\cdot \beta_{\lambda_2} \cdot \beta_{\lambda_3}$ contains the product of $\beta_{\lambda_1-2\ceil{\lambda_1}}\cdot \beta_{\lambda_2} \cdot \beta_{\lambda_3}$ if $\ceil{\lambda_1+\lambda_3} < \ceil{\lambda_1}$, and it contains $\beta_{\lambda_1-2\ceil{\lambda_1+1}}\cdot \beta_{\lambda_2} \cdot \beta_{\lambda_3}$ if $\ceil{\lambda_1+\lambda_3} = \ceil{\lambda_1}$. From this we get that all definite spaces have levels at most $\ceil{\lambda_1+1}$, so we only need to focus on this finite set of spaces. Our strategy here is to write $\beta_{\lambda_1}\cdot \beta_{\lambda_2} \cdot \beta_{\lambda_3}$ as a sum of powers of $e$ and use Lemma \ref{lem:edecomp1} (see Appendix \ref{app:p1n2class}) to determine the definite spaces. We begin by writing out the $e$-decomposition for the first $\ceil{\lambda_1+1}$ powers of $e$ in $(\beta_{\lambda_1} \beta_{\lambda_2} \beta_{\lambda_3})^-$:
\begin{lemma}\label{lem:edecompnpp} If $k \leq \ceil{\lambda_2}$, then the coefficient of $e^{\lambda-2k}$ in the $(\beta_{\lambda_1} \beta_{\lambda_2} \beta_{\lambda_3})^-$ is $\ceil{\frac{k+1}{2}}$. If $\ceil{\lambda_1} \geq k > \ceil{\lambda_2}$, then the coefficient of $e^{\lambda-2k}$ is $\ceil{\frac{k+1}{2}}$ if $2 | (k-\ceil{\lambda_2})$ and $\ceil{\frac{k+1}{2}}-(k-\floor{\lambda_2})$ otherwise.
\end{lemma}
\begin{proof} Use the fact that each term of $e^{\lambda-2k}$ in the final product arises from a sum of products of $e^{\lambda_1-2i_1}, e^{\lambda_2-2i_2}, e^{\lambda_3-2i_3}$ for $i_1+i_2+i_3 = k$. There are $\binom{k+2}{2}$ such products, and we can write them all out. If $k \leq \ceil{\lambda_2}$, these terms are easy to analyze. If $k > \ceil{\lambda_2}$, we must write out four cases based on the parities of $k$ and $\ceil{\lambda_2}$. After writing down the answer in each case, it is easy to see that the description provided covers all cases.
\end{proof}

\begin{corollary}\label{cor:npp1} Define a function $r$ by $r(i) = 0$ if $i$ is even and $r(i) = 1$ if $i$ is odd. Define a function $t$ by $t(i) = 1$ if $i$ positive and $t(i) = 0$ if $i$ nonpositive. Define a function $f$ by $$f(\lambda_1, \lambda_2, k) = \ceil{\frac{k+1}{2}} - (k-\floor{\lambda_2})\cdot r(k-\ceil{\lambda_2})\cdot t(k-\ceil{\lambda_2})$$ if $0 \leq k \leq \ceil{\lambda_1+1}$ and $f(\lambda_1, \lambda_2, k) = 0$ if $k$ is not in that range. Then an equivalent statement of Lemma \ref{lem:edecompnpp} is that the coefficient of $e^{\lambda-2k}$ in the $(\beta_{\lambda_1} \beta_{\lambda_2} \beta_{\lambda_3})^-$ is $f(\lambda_1, \lambda_2, k)$ for $0 \leq k \leq \ceil{\lambda_1}$.
\end{corollary}
\begin{corollary} By the same reasoning as in Lemma \ref{lem:edecompnpp}, we get that the coefficient of $e^{\lambda-2\ceil{\lambda_1+1}}$ in the $(\beta_{\lambda_1} \beta_{\lambda_2} \beta_{\lambda_3})^-$ is $f(\lambda_1, \lambda_2, \ceil{\lambda_1+1}) - 2$.
\end{corollary}
We now have an explicit description of the original signature character product written as a sum of powers of $e$. We will now find the definite spaces (which have levels $k$ between 0 and $\ceil{\lambda_1+1}$), tackling the $k \leq \ceil{\lambda_1}$ and $k = \ceil{\lambda_1+1}$ cases separately. We begin by solving the $k \leq \ceil{\lambda_1}$ case. Using Lemma \ref{lem:edecomp1}, the fact that the function $x \to x-2\ceil{x}$ is cyclic, and a division of the interval $[0, \ceil{\lambda_1+1}]$ into several intervals based on the signs of $\lambda-2k$ and $2k+1-\lambda$, we get the following result (Lemma \ref{lem:check}):

\begin{lemma}\label{lem:check} We have the following results on the signatures of the first $\ceil{\lambda_1+1}$ spaces. 
\begin{itemize}
\item If $0 \leq k \leq \ceil{\frac{\lambda}{2}}$, then the signature of the level $k$ space is 
$$
f(\lambda_1, \lambda_2, k) - f(\lambda_1, \lambda_2, k-1).
$$
\item If $\ceil{\frac{\lambda+2}{2}} \leq k \leq \text{min}(\ceil{\lambda_1}, \ceil{\lambda})$, then the signature of the level $k$ space is
$$
f(\lambda_1, \lambda_2, k) + f(\lambda_1, \lambda_2, k-1) - 2f(\lambda_1, \lambda_2, \ceil{\lambda}-k).
$$
\item If $\ceil{\lambda+1} \leq k \leq \ceil{\lambda_1}$, then the signature of the level $k$ space is 
$$
f(\lambda_1, \lambda_2, k) + f(\lambda_1, \lambda_2, k-1)
$$
\end{itemize}
\end{lemma}

Now, note that checking the positive/negative definiteness of each level $k$ space is equivalent to checking whether the signature of the level $k$ space is equal to $k+1$ or $-k-1$. This follows because the dimension of each level $k$ space is $\binom{k+3-2}{3-2} = k+1$, and thus the coefficient of the $\beta_{\lambda_1+\lambda_2+\lambda_3-2k}$ term (evaluated at $s = -1$) in the signature character decomposition of $\beta_{\lambda_1} \cdot \beta_{\lambda_2} \cdot \beta_{\lambda_3}$ is $k+1$ if and only if the level $k$ space is positive definite, and $-k-1$ if and only if the level $k$ is negative definite.

\begin{lemma}\label{lem:nppfirstpos} If $1 \leq k \leq \ceil{\lambda_1}$, then the level $k$ space is positive definite iff $\ceil{\lambda+1} \leq k \leq \ceil{\lambda_2}$. In particular if $\ceil{\lambda_1+1} > \ceil{\lambda_2}$, then the only positive definite space in the first $\ceil{\lambda_1+1}$ spaces is the level $0$ space.
\end{lemma}
\begin{proof} If $\ceil{\lambda+1} \leq k \leq \ceil{\lambda_2}$, then Corollary \ref{cor:npp1} and Lemma \ref{lem:check} give that the level $k$ space is definite. Now we show that no other $\ceil{\lambda_1} \geq k > 0$ has level $k$ space positive definite. If $1 \leq k \leq \ceil{\frac{\lambda}{2}}$, then using Lemma \ref{lem:check}, the maximum possible value for the signature of the level $k$ space is $$1 + (k-1) - \floor{\lambda_2} < k+1.$$ If $\ceil{\frac{\lambda+2}{2}} \leq k \leq \text{min}(\lambda_1, \lambda)$, then we must address several cases. If $k > \ceil{\lambda_2}$, $\ceil{\lambda} - k \leq \ceil{\lambda_2}$, or $2|\ceil{\lambda}-k-\ceil{\lambda_2}$, then some bounding shows that the signature of the level $k$ space is less than $k+1$. The other possibility is that $k > \ceil{\lambda_2}$, $\ceil{\lambda}-k > \ceil{\lambda_2}$, and $2 | \ceil{\lambda}-k -\ceil{\lambda_2}-1$. In this case the signature of the level $k$ space is equal to 
$$
k+1 - 2\ceil{\frac{\lambda-k+1}{2}} + 2\ceil{\lambda-k} - 2\floor{\lambda_2}.
$$
This level $k$ space is positive definite when
$$
2\ceil{\lambda-k} - 2\ceil{\frac{\lambda-k+1}{2}} = 2\floor{\lambda_2}.
$$
Viewing the LHS as a function of $k$, we see that as $k$ increases by 2, the LHS decreases by 2. Evaluating the LHS at $k=0$ and $k=1$ gives that the above equality holds only if $k$ is equal to one of 
$$
\frac{\ceil{\lambda-1}}{2} - \floor{\lambda_2}
$$
or
$$
\frac{\ceil{\lambda-2}}{2} - \floor{\lambda_2}.
$$
But $k$ can't equal either of these as $k \geq \frac{\ceil{\lambda+2}}{2}$. So there are no positive definite spaces of level $k$ when $\ceil{\frac{\lambda+2}{2}} \leq k \leq \text{min}(\lambda_1, \lambda)$. \\
If $\ceil{\lambda_2+1} \leq k \leq \ceil{\lambda_1}$, then the maximum possible value of the signature of the level $k$ space is
$$
k+1 - (k-\floor{\lambda_2}) < k+1.
$$
So, the only positive definite spaces are the ones in the given range. 
\end{proof}

\begin{lemma}\label{lem:nppfirstneg} There are no negative definite spaces with level at most $\ceil{\lambda_1}$. 
\end{lemma}
\begin{proof} Consider a level $k$ space with $0 \leq k \leq \ceil{\lambda_1}$. If $k \leq \ceil{\frac{\lambda}{2}}$, then the minimum possible signature of the level $k$ space is:
$$
0 - (k-\floor{\lambda_2}) > -k-1
$$
If $\ceil{\frac{\lambda+2}{2}} \leq k \leq \text{min}(\ceil{\lambda_1}, \ceil{\lambda})$, then the minimum possible value of the signature of the level $k$ space is
\begin{align*}
(k+1) - 2\ceil{\frac{\lambda-k+1}{2}} - k+\floor{\lambda_2} &\geq 1 - \ceil{\lambda-k+2} \\
&\geq k-1 - \ceil{\lambda} \\
&\geq -k.
\end{align*}
If $\ceil{\lambda+1} \leq k \leq \ceil{\lambda_1}$, then the minimum possible value of the signature of the level $k$ space is
$$
k+1 - (k-\floor{\lambda_2}) > -k-1.
$$
So there are no negative definite spaces.
\end{proof}
Lemmas \ref{lem:nppfirstpos} and \ref{lem:nppfirstneg} classify definite spaces with level at most $\ceil{\lambda_1}$. All that remains is to determine when the level $\ceil{\lambda_1+1}$ space is definite. 
\begin{lemma}\label{lem:excnpp} The only tensors with the level $\ceil{\lambda_1+1}$ space definite are tensors with explicit type $\langle 1,0,0,-1 \rangle$. In this explicit type, the level $2$ space is negative definite.
\end{lemma}
\begin{proof} First we check when the $\ceil{\lambda_1+1}$ space is positive definite. As $\ceil{\lambda} - \ceil{\lambda_1} \leq \ceil{\lambda_2}$ the maximum possible value of the signature of the level $\ceil{\lambda_1}$ space is:
$$
f(\lambda_1, \lambda_2, \ceil{\lambda_1+1})-2 + f(\lambda_1, \lambda_2, \ceil{\lambda}) \leq \ceil{\lambda_1} < \ceil{\lambda_1+2}.
$$
So this space can never be positive definite. For negative definiteness, we have that the minimum possible value of the signature is:
$$
f(\lambda_1, \lambda_2, \ceil{\lambda_1+1})-2 + f(\lambda_1, \lambda_2, \ceil{\lambda}) - 2f(\lambda_1, \lambda_2, \ceil{\lambda} - \ceil{\lambda_1}),
$$
which is at least $\ceil{\lambda_1} - \ceil{\lambda} - 2 + \floor{\lambda_2}$. For this space to be negative definite, then, we must have
$$
-\ceil{\lambda_1} \geq \ceil{\lambda_1} - \ceil{\lambda} + \floor{\lambda_2}
$$
Rearranging and using $\ceil{\lambda} \leq \ceil{\lambda_1}+\ceil{\lambda_2}+1$, we get that for this space to be negative definite, we must have $2 \geq \ceil{\lambda_1}$. Checking the (finitely many) explicit types with this condition yields the given exception. 
\end{proof}

\begin{proof}[Proof of $\lambda > 0$ subcase of Case 4 for $n = 3$] From initial observations, only spaces with level at most $\ceil{\lambda_1+1}$ can be definite. Lemmas \ref{lem:nppfirstpos} and \ref{lem:nppfirstneg} classify definite spaces with level at most $\ceil{\lambda_1}$, while Lemma \ref{lem:excnpp} shows when the level $\ceil{\lambda_1+1}$ is definite. Pulling these claims together gives the $\lambda > 0$ subcase of Case 4.
\end{proof}

Subsections \ref{subsec:ln} and \ref{subsec:lp} together yield the classification in Case 4. 

\section{Proof of Case 5 for $n = 3$}\label{app:pppclass}
An easy corollary to the previously proven Case 2 classification is that the first $\ceil{\lambda_3+1}$ multiplicity spaces are indeed positive definite. Thus, we need only classify multiplicity spaces with level greater than $\ceil{\lambda_3}$, which we will call the exceptional multiplicity spaces. To do this, we first determine all possible exceptional spaces in Subsections \ref{subsec:pppepl2} and \ref{subsec:pppepg2}. Then in Subsection \ref{subsec:pppexc} we use character computations to determine which of those instances actually produce exceptional spaces. \\
We begin with some notation. Write $A = \floor{\lambda_1}$, $a = \lambda_1-\floor{\lambda_1}$, and similarly define $B, b$ and $C, c$ for $\lambda_2$ and $\lambda_3$ respectively. Define $\epsilon = a+b+c$. We divide into subcases based on the floor of $\epsilon$, which is between $0$ and $2$ (inclusive).
\subsection{Possible exceptional multiplicity spaces when $\epsilon < 2$}\label{subsec:pppepl2}
\begin{lemma}\label{lem:pppepl2} If $M$ does not have explicit type $\langle 3d-1,d,d,d-1 \rangle$, $\langle 3d-2,d,d-1,d-1 \rangle$, $\langle 3d, d,d,d \rangle$, or $\langle 3d+1, d,d,d \rangle$, then it has no definite spaces with level more than $\ceil{\lambda_3}$. If $M$ has one of those types, then the possible exceptional definite spaces have levels $2d+1$ (first type), $2d$ (second type), $2d+1$ and $2d+2$ (third type), and $2d+2$ (fourth type).
\end{lemma}
\begin{proof}
By "fudging" the fractional parts of $\lambda_1$, $\lambda_2$, $\lambda_3$, we can ensure that $b+c < 1$ while maintaining $M$'s explicit type. If this condition is satisfied, then by the previously proven Case 2 classification, the signature character product $\beta_{\lambda_1} \cdot \beta_{\lambda_2} \cdot \beta_{\lambda_3}$ contains $\beta_{\lambda_1} \cdot \beta_{\lambda_2} \cdot \beta_{\lambda_3-2\ceil{\lambda_3}}$. From the closed form of two positive and one negative Vermas, we obtain that the only spaces with level more than $\ceil{\lambda_3}$ that can be definite in $M$s are spaces with level $k$, where $k$ satisfies
$$
\ceil{\lambda_1+\lambda_2+\lambda_3-2\ceil{\lambda_3}+1}+\ceil{\lambda_3} \leq k \leq \ceil{\lambda_2}+\ceil{\lambda_3}.
$$
This inequality is equivalent to
$$
\ceil{A+B-C+\epsilon-1}+\ceil{\lambda_3} \leq k \leq B+1+\ceil{\lambda_3}.
$$
Using some case analysis and bounding, we obtain that the only possible exceptional definite spaces occur in types:
\begin{itemize}
\item Explicit type $\langle 3d-1,d,d,d-1 \rangle$, in which the level $2d+1$ space could be definite.
\item Explicit type $\langle 3d-2,d,d-1,d-1 \rangle$, in which the level $2d$ space could be definite.
\item Explicit type $\langle 3d, d,d,d \rangle$, in which the level $2d+1$ and $2d+2$ spaces could be definite.
\item Explicit type $\langle 3d+1, d,d,d \rangle$, in which the level $2d+2$ space could be definite.
\end{itemize}
\end{proof}

\subsection{Possible exceptional multiplicity spaces when $2 < \epsilon < 3$}\label{subsec:pppepg2}
\begin{lemma}\label{lem:pppepg2} If $M$ does not have explicit type $\langle 3d,d,d-1,d-1 \rangle$, $\langle 3d+2,d,d,d \rangle$, or $\langle 2d+d'+2, d,d,d' \rangle$ for $d > d'$, then it has no definite spaces with level more than $\ceil{\lambda_3}$. If $M$ has one of those types, then the possible exceptional definite spaces have levels $2d+1$ (first type), $2d+2$ and $2d+3$ (third type), and $d+d'+3$ (fourth type).
\end{lemma}
\begin{proof} We divide into cases based on whether $\floor{\lambda_3} = \floor{\lambda_2}$ or not. \\
Suppose $\floor{\lambda_2} = \floor{\lambda_3}$. Then from the previously proven Case 2 classification, the signature character product $\beta_{\lambda_1} \cdot \beta_{\lambda_2} \cdot \beta_{\lambda_3}$ contains $\beta_{\lambda_1} \cdot \beta_{\lambda_2-2\ceil{\lambda_3}-2} \cdot \beta_{\lambda_3}$. By the definite space classification for the previously proven $n=3$ subcase of Case 4, the only possible exceptional definite spaces are those with level $k$ where $k$ satisfies
$$
\ceil{\lambda_1+\lambda_2+\lambda_3-2\ceil{\lambda_2}-2+1} + \ceil{\lambda_3+1} \leq k \leq \ceil{\lambda_3} + \ceil{\lambda_3+1}.
$$
This inequality is equivalent to
$$
A+\ceil{\lambda_3+1} \leq k \leq C+1+\ceil{\lambda_3+1}.
$$
Using some case analysis and bounding, we get that the only possible exceptional definite spaces occur in
\begin{itemize}
\item $\langle 3d,d,d-1,d-1 \rangle$, in which the level $2d+1$ space could be definite.
\item $\langle 3d+2, d,d,d \rangle$, in which the level $2d+2$ and $2d+3$ spaces could be definite.
\end{itemize}
The other possibility is that $\floor{\lambda_3} < \floor{\lambda_2}$. Using the previously proven Case 2 classification, the signature character product $\beta_{\lambda_1} \cdot \beta_{\lambda_2} \cdot \beta_{\lambda_3}$ contains $\beta_{\lambda_1} \cdot \beta_{\lambda_2-2\ceil{\lambda_2}-1} \cdot \beta_{\lambda_3+1}$. Using the definite space classification for the previously proven $n=3$ subcase of Case 4, we get that the only possible exceptional definite spaces are spaces with level $k$ where $k$ satisfies
$$
\ceil{\lambda_1+\lambda_2+\lambda_3-2\ceil{\lambda_2}+1} + \ceil{\lambda_2} \leq k \leq \ceil{\lambda_3+1}+\ceil{\lambda_2}
$$
Using some case analysis and bounding, we get the only possible exceptional definite spaces occur in $\langle 2d+d'+2, d,d,d' \rangle$, in which the $d+d'+3$ space could be definite. Here $d > d'$.
\end{proof}

\subsection{Classification of exceptions}\label{subsec:pppexc}
Here we will address the exceptions found in the previous two subsections and show that six of the seven potential families of exceptions are always exceptional, and the seventh family is never exceptional. We begin with the seventh family.
\begin{lemma} In explicit type $\langle 3d+1, d,d,d \rangle$, the level $2d+2$ space is nondefinite.
\end{lemma}
\begin{proof} Fudge around the fractional parts of $\lambda_1$, $\lambda_2$, and $\lambda_3$, so that $a+c > 1$ and $b+c < 1$. (We can do this and maintain the explicit type of $M$.) By multiplying $\beta_{\lambda_1} \cdot \beta_{\lambda_3}$ and using the previously proven Case 2 classification, we get that the level $2d+2$ space is not positive definite. By multiplying $\beta_{\lambda_2} \cdot \beta_{\lambda_3}$ and using the Case 2 classification, we get that the level $2d+2$ space is not negative definite. We are done.
\end{proof}

Now we will show that the other six exceptional families are always exceptional, and we will determine the positivity/negativity of the exceptional definite spaces. Our strategy here is to write the signature character product as a sum of powers of $e$ and use the following results along with Lemma \ref{lem:edecomp1} to compute the signatures of the multiplicity spaces:

\begin{definition*} Define a function $f_2$ of a real $x$ and an integer $j$ by
$$
f_2(x,j) = 0
$$
if $j \leq \ceil{x}$ and
$$
f_2(x,j) = (j-\ceil{x}-\floor{\frac{j-\ceil{x+1}}{2}})(1+\floor{\frac{j-\ceil{x+1}}{2}})
$$
if $j \geq \ceil{x+1}$.
\end{definition*}

\begin{lemma}\label{lem:pppedecomp2} For $k \leq \ceil{\lambda_2}+\ceil{\lambda_3}+1$, the coefficient of $e^{\lambda-2k}$ in the original signature character product (evaluated at $s=-1$) is
$$
\binom{k+2}{2} - 2f_2(\lambda_1, k) - 2f_2(\lambda_2, k) - 2f_2(\lambda_3, k)
$$
\end{lemma}
\begin{proof} Each term of $e^{\lambda-2k}$ arises from a product of $e^{\lambda_1-2i_1}$, $e^{\lambda_2-2i_2}$, and $e^{\lambda_3-2i_3}$, where $i_1+i_2+i_3=k$. If all of these terms were $1$, then by a standard counting argument, the coefficient would be $\binom{k+2}{2}$. However, some of the terms are $-1$. So if we can count the number of terms that are $-1$ and subtract twice that from $\binom{k+2}{2}$, we will be done. \\
A $-1$ arises as a product of $-1\cdot1\cdot1$ or $-1\cdot-1\cdot-1$. But since $k$ is bounded by $\ceil{\lambda_3}+\ceil{\lambda_2}+1$, the three $-1$'s cannot happen. So we need only count the number of $-1,1,1$ triples. \\
We claim that the number of valid triples with a $-1$ coming from $\beta_{\lambda_1}$ is $f_2(\lambda_1, k)$. To see this, count the number of solutions to:
$$
2i'_1 + i_2+i_3 = k-\ceil{\lambda_1+1}
$$
using standard methods. By symmetry, the total number of valid triples is $f_2(\lambda_1, k) + f_2(\lambda_2, k) + f_2(\lambda_3, k)$. Subtracting twice this from our original count gives the result.
\end{proof}
\begin{lemma} We have the following exceptional definite spaces:
\begin{itemize}
\item For $d \geq 0$ and explicit type $\langle 3d,d,d,d \rangle$, the level $2d+1$ and $2d+2$ spaces are positive definite.
\item For $d \geq 0$ and explicit type $\langle 3d+2,d,d,d \rangle$, the level $2d+2$ and $2d+3$ spaces are negative definite.
\item For $d \geq 1$ and explicit type $\langle 3d-1,d,d,d-1 \rangle$, the level $2d+1$ space is positive definite.
\item For $d \geq 1$ and explicit type $\langle 3d+1, d,d, d-1 \rangle$, the level $2d+2$ space is negative definite.
\item For $d \geq 1$ and explicit type $\langle 3d-2,d,d-1,d-1 \rangle$, the level $2d$ space is positive definite.
\item For $d \geq 1$ and explicit type $\langle 3d,d,d-1,d-1 \rangle$, the level $2d+1$ space is negative definite.
\end{itemize} 
\end{lemma}
\begin{proof} We will outline the proof for $\langle 3d, d,d,d \rangle$. The other cases are exactly the same. \\
For $\langle 3d,d,d,d \rangle$, we know the only possible exceptional definite spaces are levels $2d+1$ and $2d+2$. We claim that these are both positive definite. Based on whether $d$ is odd or even, the claim that in $\langle 3d,d,d,d \rangle$, levels $2d+1$ and $2d+2$ are positive definite is by Lemmas \ref{lem:edecomp1}, \ref{lem:pppedecomp2} and the cyclicity of the map $x \to x-2\ceil{x}$, equivalent to some quadratic polynomials being identically 0. 

For example, take the case of even $d$ (so $d = 2a$). Writing $\lambda = \lambda_1+\lambda_2+\lambda_3$, we see that by Lemma \ref{lem:edecomp1} that \begin{align*} e^{\lambda-4d-2} &= \beta_{\lambda - 4d - 2}^- + \beta_{\lambda - 4d - 4}^-, \\ e^{\lambda-4d} &= \beta_{\lambda - 4d}^- + \beta_{\lambda - 4d - 2}^-, \\ e^{\lambda - 2\ceil{\lambda}+ 4d + 2} &= \beta_{\lambda - 2\ceil{\lambda}+ 4d + 2}^- - \beta_{\lambda - 2\ceil{\lambda}+ 4d}^- - 2\beta_{\lambda - 4d - 2}^-. \end{align*}
Using Lemma \ref{lem:pppedecomp2}, the coefficients of these three powers of $e$ in the signature character product corresponding to our tensor product are 
\begin{align*}
&\binom{2d+3}{2} - 2f_2(\lambda_1, 2d+1) - 2f_2(\lambda_2, 2d+1) - 2f_2(\lambda_3, 2d+1), \\
&\binom{2d+2}{2} - 2f_2(\lambda_1, 2d) - 2f_2(\lambda_2, 2d) - 2f_2(\lambda_3, 2d), \\
&\binom{d+2}{2} - 2f_2(\lambda_1, d) - 2f_2(\lambda_2, d) - 2f_2(\lambda_3, d),
\end{align*}
respectively. As an easy corollary of Lemma \ref{lem:edecomp1} we know that $\beta_{\lambda - 4d - 2}$ only arises in the signature decompositions of $e^{\lambda-4d-2}$, $e^{\lambda-4d}$, and $e^{\lambda - 2\ceil{\lambda}+ 4d + 2}$, so we see that the signature of the level $2d+1$ space is 

\begin{align*}
&\Big( \binom{2d+3}{2} - 2f_2(\lambda_1, 2d+1) - 2f_2(\lambda_2, 2d+1) - 2f_2(\lambda_3, 2d+1) \Big) \\
+ &\Big(\binom{2d+2}{2} - 2f_2(\lambda_1, 2d) - 2f_2(\lambda_2, 2d) - 2f_2(\lambda_3, 2d) \Big) \\
- 2&\Big(\binom{d+2}{2} - 2f_2(\lambda_1, d) - 2f_2(\lambda_2, d) - 2f_2(\lambda_3, d) \Big).
\end{align*}

Since by the definition of $f_2$ we know that for any real $x$ that $f_2(x, j) = f_2(\ceil{x}, j)$, we get that the signature of the level $2d+1$ is equal to:

\begin{align*}
\Big( \binom{2d+3}{2} - 6f_2(d+1, 2d+1) \Big)
+ \Big(\binom{2d+2}{2} - 6f_2(d+1, 2d) \Big)
- 2\Big(\binom{d+2}{2} - 6f_2(d+1, d) \Big).
\end{align*}

Now, our claim that the level $2d+1$ space is positive definite is the same as claiming that the signature of the level $2d+1$ space is equal to the dimension $2d+2$ of the level $2d+1$ space, which is equivalent to saying that:

\begin{align*}
&\Big( \binom{4a+3}{2} - 6f_2(2a+1, 4a+1) \Big) \\
+ &\Big(\binom{4a+2}{2} - 6f_2(2a+1, 4a) \Big) \\
- 2&\Big(\binom{2a+2}{2} - 6f_2(2a+1, 2a) \Big) \\
- 4&a - 2
\end{align*}
is identically zero for all nonnegative integers $a$ (above, we have substituted in $d = 2a$). Expanding the above expression completely shows that 
\begin{align*}
0 = &\Big( \binom{4a+3}{2} - 6f_2(2a+1, 4a+1) \Big) \\
+ &\Big(\binom{4a+2}{2} - 6f_2(2a+1, 4a) \Big) \\
- 2&\Big(\binom{2a+2}{2} - 6f_2(2a+1, 2a) \Big) \\
- 4&a - 2,
\end{align*}
as needed, so the level $2d+1$ space is indeed positive definite for even $d$. Completing the analogous calculations for the level $2d+1$ space for odd $d$ and the level $2d+2$ space for odd and even $d$ shows that the level $2d+1$ and $2d+2$ spaces are always positive definite in explicit type $\langle 3d, d, d, d \rangle$, as desired.

\end{proof}
\begin{lemma} In explicit type $\langle 2d+d'+2, d,d,d' \rangle$ where $d > d'$, the only potential exceptionally definite space (level $d+d'+2$) is definite exactly when $d'=d-1$. If it is definite then it is negative definite.
\end{lemma}
\begin{proof} Set $a = \ceil{n/2}$ and $b = \ceil{m/2}$. Using Lemmas \ref{lem:edecomp1} and \ref{lem:pppedecomp2} we get that the signature of the level $d+d'+2$ space is negative, so we only need to check when it is equal to $-d-d'-3$. If $d$ and $d'$ have oppositive parities, than using Lemmas \ref{lem:edecomp1} and \ref{lem:pppedecomp2} gives that this happens exactly when $d'=d-1$. If $d$ and $d'$ have the same parity, this happens either when $d=d'$ (impossible) or never. 
\end{proof}
From the classification of exceptional multiplicity spaces proved in Subsection \ref{subsec:pppexc}, we obtain Case 5.

\section{Proof of Remaining Cases}\label{app:n4class}
We start with some new definitions. Define a function $\Lambda$ by $\Lambda(k) = \sum_{i=1}^k \lambda_i$. Let $\lambda_+ = \Lambda(p)$ and $\lambda_- = \lambda - \lambda_+$.

\subsection{Proof of Case 6}
\begin{lemma}\label{lem:pn21} If $p \geq 2$, $n-p \geq 2$, and $|\lambda_+| > |\lambda_{p+1}|$, then there is exactly one definite space in $M$. It is the level $0$ space and it is positive definite.
\end{lemma}
\begin{proof} When we multiply $\beta_{\lambda_1}, \beta_{\lambda_2}, ..., \beta_{\lambda_{p+1}}$, the product is:
$$
e^{\lambda_{p+1}+\lambda_+} + (p+s)e^{\lambda_{p+1}+\lambda_+-2} = \beta_{\lambda_{p+1}+\lambda_+} + (p-1 + s)\beta_{\lambda_{p+1}+\lambda_+-2}
$$
Since the level $1$ space is nondefinite, when we multiply by $\beta_{\lambda_{p+2}}$, we get that all spaces with level at least $1$ are nondefinite. Thus, in $M$ product, all spaces with level at least $1$ are nondefinite. In $M$ product, the level $0$ space is positive definite. Our claim follows.
\end{proof}

\begin{lemma}\label{lem:pn22} If $p \geq 1$, $n-p \geq 2$, and $|\lambda_-| > |\lambda_p|$, then there is exactly one definite space in $M$. It is the level 0 space and it is positive definite.
\end{lemma}

\begin{proof}
The product of the $n-p-1$ signature characters $\beta_{\lambda_{p+2}},..., \beta_{\lambda_n}$ contains $\beta_{-\lambda_{p+1} +\lambda_-}$. So the product of $\beta_{\lambda_p}, \beta_{\lambda_{p+1}}, ... \beta_{\lambda_n}$ contains the product of $\beta_{\lambda_p}, \beta_{\lambda_{p+1}}, \beta_{-\lambda_{p+1}+\lambda_-}$. From the previously proven Case 3 classification, we deduce that there is exactly one definite space. That is the level $0$ space which is positive definite.
\end{proof}

\begin{lemma} If $p \geq 2$ and $n-p \geq 2$, then there is exactly one definite space in $M$. It is the level 0 space and it is positive definite.
\end{lemma}

\begin{proof} Assume the claim is false. Then by Lemma \ref{lem:pn21}, $\lambda_{p+1} + \lambda_+ < 0$. By Lemma \ref{lem:pn22}, $0 < \lambda_p + \lambda_-$. Adding these yields
$$
\lambda_+ + \lambda_{p+1} < \lambda_p + \lambda_-
$$
or
$$
\sum_{i<p} \lambda_i < \sum_{i>p+1} \lambda_i
$$
which is a contradiction as the LHS is positive and the RHS is negative. So our assumption was false and the claim holds. \end{proof}
This proves the Case 6 classification. 

\subsection{Proof of Case 4 for $n \geq 4$} 
\begin{lemma} If $|\Lambda(p-1)| > |\lambda_n|$, then there is exactly one definite space in $M$. It is the level 0 space and it is positive definite. \end{lemma}

\begin{proof} We know $p \geq 3$ since $n \geq 4$. By the same logic as in the proof of Lemma \ref{lem:pn21}, in the tensor product of $M_{\lambda_1} \otimes \cdots \otimes M_{\lambda_{p-1}} \otimes M_{\lambda_n}$, the level 1 space is nondefinite. So when we tensor this with $M_{p}$, all the spaces with level at least 1 are nondefinite. The level 0 space is positive definite. The lemma follows.
\end{proof}

\begin{lemma}\label{lem:pnminus11} If $|\Lambda(p-1)| < |\lambda_n|$, then no space after the level $\ceil{\lambda_p}$ is definite in $M$. \end{lemma}
\begin{proof} 
The product of $\beta_{\lambda_p}$ and $\beta_{\lambda_n}$ contains the product of $\beta_{\lambda_p - 2\ceil{\lambda_p}}$ and $\beta_{\lambda_n}$. Thus, the product of all our $n$ signature characters contains the product of $\beta_{\lambda_1},..., \beta_{\lambda_{p-1}}, \beta_{\lambda_p - 2\ceil{\lambda_{p}}}, \beta_{\lambda_n}$. Applying Lemma \ref{lem:pn22} gives that in $M$ product, no space after the level $\ceil{\lambda_p}$ space is definite.
\end{proof}

\begin{lemma} If $|\Lambda(p-1)| < |\lambda_n| < |\Lambda(p)|$, then there are exactly $\ceil{\lambda_p} - \floor{\lambda}$ definite spaces in $M$. They are all positive definite; one of them is level 0 and the others are levels $\ceil{\lambda+1}$ through $\ceil{\lambda_p}$.
\end{lemma}

\begin{proof}
When we multiply $\beta_{\lambda_1},..., \beta_{\lambda_p}$, it's easy to see that the first $\ceil{\lambda_p}$ signature characters in the sum decomposition are positive definite. (Use the fact that $s$'s don't appear in the series for each $\beta_x$ until $\ceil{x+1}$ when $x$ is positive.) Using the fact that the product of $\beta_{\lambda_1}, ... ,\beta_{\lambda_p}$ contains $\beta_{\lambda_+}$ and $p\beta_{-2+\lambda_+}$, along with the Case 2 classification and Lemma \ref{lem:pnminus11}, gives that the definite spaces in $M$ product are positive and have levels 0 and $\ceil{\lambda+1}$ through $\ceil{\lambda_p}$.
\end{proof}

\begin{lemma} If $|\lambda_n| > \Lambda(p)$, then there are exactly $\ceil{\lambda_p+1}$ definite spaces in $M$. They are all positive definite and they have levels 0 through $\ceil{\lambda_p}$. \end{lemma} 

\begin{proof}
From Lemma \ref{lem:pnminus11}, it's enough to show that spaces with level 0 to $\ceil{\lambda_p}$ are positive definite. When we multiply $\beta_{\lambda_1},..., \beta_{\lambda_p}$, it's easy to see that the first $\ceil{\lambda_p}$ signature characters in the sum decomposition are positive definite. (Use the fact that $s$'s don't appear in the series for each $\beta_x$ until $\ceil{x+1}$ when $x$ is positive.) Using this, along with the previously proven Case 2 classification, gives that the first $\ceil{\lambda_p +1}$ spaces are definite in $M$ product.
\end{proof}

Combining the results of this subsection gives the $n \geq 4$ subcase of Case 4. The form of the classification stated in Case 4 is slightly different than the form of the classification we have proved here, but it can be easily shown to be equivalent.

\subsection{Proof of Case 5 for $n \geq 4$}
\begin{lemma}\label{lem:speccond} Suppose $M$ does not satisfy either of the following special conditions:
\begin{itemize}
\item $p > 4$, and $M$ has implicit type either $\langle 1, 0, ..., 0 \rangle$ or $\langle 0, 0, ... , 0 \rangle$.
\item $p = 4$, and $M$ has $\floor{\lambda_1} \in \{ 0, 1 \}$.
\end{itemize}
Then $M$ has exactly $\ceil{\lambda_p+1}$ definite spaces. They are all positive definite and they have levels 0 through $\ceil{\lambda_p}$.
\end{lemma}
\begin{proof}
Since $M$ contains $M_{\lambda_1} \otimes M_{\lambda_2} \otimes M_{\lambda_n}$ as a sub-tensor product, the already proven $n=3$ subcase of Case 5 gives that the only possible implicit types with exceptional definite spaces are $\langle 0, 0, ..., 0 \rangle$, $\langle 1, 1, ... , 1 \rangle$, and $\langle 1, 0, ..., 0 \rangle$. None of the (finitely many) explicit types for $n=5$ that have implicit type $\langle 1, 1, ..., 1 \rangle$ has exceptional definite spaces, so the claim is proven.
\end{proof}

\begin{lemma}\label{lem:pnminus12} Suppose $p \geq 4$ and $M$ has explicit type $\langle 0,0,...,0 \rangle$. Then it has exactly three definite spaces.
\end{lemma}
\begin{proof} It is easily checked that in a tensor product decomposition of three Verma modules of the original set of Verma modules, the level 3 space has multiplicity signature $1+3s$. Thus no space after level 2 is definite in $M$. It is then easily checked that levels 0, 1,  and 2 are definite.
\end{proof}

\begin{lemma} Suppose $p > 4$ and $M$ has implicit type $\langle 0,0, ..., 0\rangle$. Then if $\lambda_+ < 1$, $M$ has exactly 3 definite spaces: They are levels 0,1,2 and they are positive definite. If $\lambda_+ > 1$, then $M$ has exactly 2 definite spaces: They are levels 0,1 and they are positive definite.
\end{lemma}
\begin{proof} If $\lambda_+ < 1$, then the claim follows from Lemma \ref{lem:pnminus12}. Otherwise suppose $\lambda_+ > 1$. Let $i$ be minimal such that $\Lambda(i) > 1$. From assumptions, $2 \leq i \leq p$. If $2 < i < p$, then in the tensor product of $M_{\lambda_1},\ldots, M_{\lambda_i}$, we can check that the level 2 space is nondefinite, so our claim holds for $M$. Otherwise we have $i = p$ or $i = 2$. If $i = p$, tensor the first $p-1$ Vermas and apply Lemma \ref{lem:pnminus12} and the already proven Case 2 classification to get the claim. If $i = 2$, check that all the (finitely many) explicit types for $n=5$ which have $\lambda_+ > 1$ and implicit type $\langle 0,0,0,0,0 \rangle$ have two definite spaces. So in all cases the claim holds.
\end{proof}

\begin{lemma} Suppose $p > 4$ and $M$ has implicit type $\langle 1,0,...,0 \rangle$. Then if $\lambda_+ < 2$, $M$ has exactly 3 definite spaces: They are levels 0,1,2 and they are positive definite. If $\lambda_+ > 2$, then $M$ has exactly 2 definite spaces: They are levels 0,1 and they are positive definite.
\end{lemma}
\begin{proof} Check that the claim holds for all (finitely many) explicit types of tensors which have $p=5$ and implicit type $\langle 1,0,..,0 \rangle$. For $p > 5$, tensoring the 5 Vermas with largest highest weights and using the result in the last sentence gives that no space after level 3 can be definite. Check that if $\lambda_+ < 2$ there are three definite spaces and if $\lambda_+ > 2$ there are two definite spaces as claimed.
\end{proof}

\begin{lemma} Suppose $p = 4$. Then there are $\ceil{\lambda_p+1}$ definite spaces in $M$ as described before except for the following exceptions:
\begin{itemize}
\item Explicit type $\langle 0,0,0,0,0 \rangle$, in which level 0,1,2 spaces are positive definite.
\item Explicit type $\langle 3,0,0,0,0 \rangle$, in which level 0,1 spaces are positive definite and level 3 space is negative definite.
\item Explicit type $\langle 1,1,0,0,0 \rangle$, in which level 0,1,2 spaces are positive definite.
\item Explicit type $\langle 4,1,1,1,1 \rangle$, in which the level 0,1,2,4 spaces are positive definite.
\end{itemize}
\end{lemma}
\begin{proof} From Lemma \ref{lem:speccond}, we only need to check the (finitely many) explicit types for which $\lambda_1 < 2$.
\end{proof}

Combining the results in this subsection gives the $n \geq 4$ subcase of Case 5.

\section{Proof of Lemma \ref{bQreal}}\label{app:bQreal}
Recall the definition
$
b_Q = Y(t_1)Y(t_2) \cdots Y(t_m)v
$
from Section \ref{sec:cpbprelims}. Expanding this gives the equivalent definition
$$
b_Q = \sum_{a_1+a_2+...+a_n = m} \bigotimes_{i=1}^n F^{a_i}v_i \cdot \left( \sum_{\sigma} \frac{1}{\prod_{j=1}^m (t_j - z_{\sigma(j)})} \right),
$$
where $\sigma$ ranges over all maps $\sigma:[m]\to[n]$ satisfying $|\sigma^{-1}\{i\}|=a_i$.

We assume that the joint eigenvalue (for the action of $\mathcal{H}_1,\ldots,\mathcal{H}_n$) of $b_Q$ is real. The operators $\mathcal{H}_i$ preserve the obvious real form of $\bigotimes_{i=1}^n M_{\lambda_i}$. It follows that the real and imaginary parts of $b_Q$ are also joint eigenvectors with the same joint eigenvalue. But recall (Lemma \ref{lemyo}) that the joint eigenspaces are all one-dimensional; it follows that $b_Q$ is proportional to a real vector: there exists $a\in\mathbb{C}^\times$ such that $\sum_{\sigma} \frac{1}{\prod_{j=1}^m (t_j - z_{\sigma(j)})}\in a\mathbb{R}$ for all partitions $(a_1,\ldots,a_n)$ of $m$. Considering partitions with $n-1$ zero entries, one obtains that $1/Q(z_i)\in a\mathbb{R}$ for each $i$.

We show by induction on $k$ that $Q^{(k)}(z_i)/Q(z_i)\in\mathbb{R}$ for each $i$. The case $k=1$ is the assumption that the joint eigenvalue is real, see the corollary to Lemma \ref{lem:hb}.
Assume that $Q'(z_i)/Q(z_i), ..., Q^{(k-1)}(z_i)/Q(z_i)$ are all real for each $i$. We show that $Q^{(k)}(z_i)$ is real for each $i$. Consider the partition $p_k = (m-k,k,0,...,0)$ of $m$. The coefficient of the summand of $b_Q$ indexed by $p_k$ is
$$
\frac{1}{Q(z_1)} \cdot \sum_{k\text{-subsets $S$ of } \{1,2,...,m\}} \frac{\prod_{i \in S} (t_{i}-z_1)}{\prod_{i \in S} (t_{i}-z_2)}\in a\mathbb{R};
$$
multiplying by $Q(z_1)$ we get that 
$$
\sum_{k\text{-subsets $S$ of } \{1,2,...,m\}} \frac{\prod_{i \in S} (t_{i}-z_1)}{\prod_{i \in S} (t_{i}-z_2)}\in \mathbb{R}.
$$
Define a polynomial associated to the $k$-set $S$ by $R_S(z) = \prod_{i \in S} (t_{i}-z)$. Then the line above says that
$$
\sum_{k\text{-subsets } S} \frac{R_S(z_1)}{R_S(z_2)}\in\mathbb{R}.
$$
Using the Taylor expansion
$$
\frac{R_S(z_1)-R_S(z_2)}{z_1-z_2} = R'_S(z_2) + \left( \frac{z_1-z_2}{2!} \right)R''_S(z_2) + \dots
$$
we get that
$$
\sum_{k\text{-subsets } S} \frac{R'_S(z_2) + \left( \frac{z_1-z_2}{2!} \right)R''_S(z_2) + ...}{R_S(z_2)} = \frac{1}{z_1-z_2}\sum_{k\text{-subsets } S} \frac{R_S(z_1)}{R_S(z_2)}-\frac{\binom{m}{k}}{z_1-z_2}\in\mathbb{R}.
$$
But this is equal to
\begin{align*} & \sum_{k\text{-subsets } S} \left(\sum_{\{j_1\} \subset S} \frac{1}{z_2-t_{j_1}} + (z_1-z_2) \sum_{2-\text{subsets } \{j_1,j_2\}\subset S} \frac{1}{(z_2-t_{j_1})(z_2-t_{j_2})} + \ldots\right)\\
= & \binom{m-1}{k-1}\frac{Q'(z_2)}{Q(z_2)}+\binom{m-2}{k-2}\frac{z_1-z_2}{2!}\frac{Q''(z_2)}{Q(z_2)}+\binom{m-3}{k-3}\frac{(z_1-z_2)^2}{3!}\frac{Q'''(z_2)}{Q(z_2)}+\ldots + \binom{m-k}{0}\frac{(z_1-z_2)^{k-1}}{(k)!}\frac{Q^{(k)}(z_2)}{Q(z_2)}
\end{align*}
All but the last of these summands are assumed to be real in the inductive hypothesis. Therefore the last term is also real, and rescaling we get that $\frac{Q^{(k)}(z_2)}{Q(z_2)}\in\mathbb{R}$. We similarly prove (using permutations of the partitions $p_k$) that $\frac{Q^{(k)}(z_i)}{Q(z_i)}\in\mathbb{R}$ for all $i$. This completes the induction.\\

\noindent We have shown that $Q^{(k)}(z_i)\in a^{-1}\mathbb{R}$ for all $i$, $k$. But taking $k=m$ gives $\pm m! \in a^{-1}\mathbb{R}$, so $a\in\mathbb{R}$ and hence $Q^{(k)}(z_i)\in \mathbb{R}$ for all $i$, $k$.\\

\noindent Let $P(z)=Q(z+z_1)$. Then $P$ has real coefficients if and only if $Q$ does, since $z_1$ is real. But all derivatives of $P$ evaluated at $0$ give real numbers. So $P$, and hence $Q$, has real coefficients.

\end{document}